\documentclass[final,5p,twocolumn]{elsarticle}%

\journal{Automatica}

\usepackage{graphicx}      
\usepackage{natbib}        

\usepackage{bookmark}

\usepackage{epsfig}

\usepackage{amsmath}%

\graphicspath{{fig/}}

\usepackage{amsfonts}
\usepackage{amssymb}

\usepackage{tikz}
\usetikzlibrary{matrix,positioning,decorations.pathreplacing}
\usetikzlibrary{arrows.meta,patterns}

\usepackage{color}

\usepackage{enumitem}

\usepackage{ifthen}

\newenvironment{proof}{{\bf Proof.}}{\hfill \hspace*{1pt}\hfill $\Box$}

\newtheorem{Ass}{Assumption}

\newtheorem{theorem}{Theorem}[section]
\newtheorem{lemma}[theorem]{Lemma}
\newtheorem{proposition}[theorem]{Proposition}
\newtheorem{corollary}[theorem]{Corollary}

\newtheorem{definition}[theorem]{Definition}

\newtheorem{remark}[theorem]{Remark}

\newcounter{syscounter}

\usepackage{csquotes}  

\newcommand \N   {\mathbb{N}}
\newcommand \R   {\mathbb{R}}

\newcommand \K   {\mathcal{K}}
\newcommand \Kinf{\mathcal{K_\infty}}
\newcommand \PD  {\mathcal{P}}
\newcommand \KL  {\mathcal{KL}}
\newcommand \LL  {\mathcal{L}}

\newcommand{\Uc}{\ensuremath{\mathcal{U}^m}}
\newcommand{\X}{\ensuremath{\mathcal{X}^n}}

\newcommand{\muval}{\mathfrak u}

\newcommand \qrq   {\quad\Rightarrow\quad}

\newcommand \Iff   {\Leftrightarrow}

\newcommand \eps {\varepsilon}

\newcommand{\sumform}{dissipative form}
\newcommand{\Sumform}{Dissipative form}

\newcommand{\amc}[1]{{\color{black} #1}}           

\newcommand{\ac}[1]{{\color{black} #1}} 
  
\newcommand{\fw}[1]{{\color{black} #1}} 
\newcommand{\lb}[1]{\begingroup\color{black}#1\endgroup}

\newif\ifAndo

\Andotrue

\title{\LARGE 
A superposition approach for the ISS Lyapunov-Krasovskii theorem\\ with pointwise dissipation}

\tnotetext[mytitlenote]{This work has been supported by BayFrance, project  FK-20-2022. 
A. Mironchenko has been supported by the Heisenberg program of the German Research Foundation (DFG), grant MI\,1886/3-1.}
 
\author[bayreuthaddress]{Andrii~Mironchenko\corref{cor1}}
\address[bayreuthaddress]{Department of Mathematics, University of Bayreuth, Germany (email: andrii.mironchenko@uni-bayreuth.de)}

\cortext[cor1]{Corresponding author.}

\author[passauaddress]{Fabian Wirth}
\address[passauaddress]{Faculty of Computer Science and Mathematics, University of Passau, Germany (email: firstname.lastname@uni-passau.de)}

\author[psuaddress]{Antoine Chaillet}

\author[psuaddress]{Lucas Brivadis}

\address[psuaddress]{Universit\'e Paris-Saclay, CNRS, CentraleSup\'elec, Laboratoire des signaux et syst\`emes, 91190, Gif-sur-Yvette, France (email: firstname.lastname@centralesupelec.fr)}

\begin{document}

\begin{abstract}
For nonlinear time-delay systems, we
 show that the existence of a Lyapunov-Krasovskii functional (LKF) with pointwise dissipation 
(i.e., dissipation in terms of the current solution norm)  
suffices for input-to-state stability, provided that uniform global stability can also be ensured using the same LKF. To this end, we 
employ elements of a recently developed stability theory in which the behavior of solutions is not assessed through the classical norm but rather through a specific LKF. 
As a consequence of our results, we obtain a strengthening and extension of the classical Lyapunov-Krasovskii theorem for undisturbed systems to systems with uniformly bounded disturbances.  
We discuss the advantages of our approach by means of an example.
\end{abstract}

\begin{keyword}
Nonlinear control systems, input-to-state stability, time-delay systems, infinite-dimensional systems.
\end{keyword}

\maketitle

\section{Introduction}
\label{sec:introduction}

Input-to-state stability (ISS), introduced by E.D. Sontag in the late 1980s \cite{Son89}, has become a central tool in the analysis and control of nonlinear dynamical systems \cite{Son08,Mir23}. Originally defined in the context of ordinary differential equations, it has been extended more recently to infinite-dimensional systems \cite{MiP20,KaK19}, including time-delay systems \cite{CKP23}.

For time-delay systems, ISS can be established by means of Lyapunov-Krasovskii functionals (LKFs) \cite{PeJ06}: ISS holds if the LKF dissipates along the system's solutions, modulo a positive term involving the input norm. So far, the only general conditions to ensure ISS based on LKFs require that the dissipation can be expressed in terms of the LKF itself (
\fw{the existence of LKFs with this property} is also a necessary requirement for ISS \cite{KPJ08}). 

However, the classical result due to Krasovskii \cite{Kra63, HaV93} shows that for the uniform global asymptotic stability of undisturbed delay systems, it is sufficient to have an LKF with pointwise dissipation, and thus the dissipation in terms of the LKF itself is not essential.
This motivated the authors of \cite{CPM17} to conjecture that a pointwise dissipation, involving merely the norm of the current value of the solution, is enough to guarantee ISS. \fw{In the original conception of the present paper, we, the authors, believed that a further step towards a general proof of the conjecture is achieved here. During the preparation of the first revision, however, an LLM-generated counterexample to the full conjecture has been presented in \cite{de2026candidate}. Thus, unlike for the weaker notion of integral ISS \cite{CGP22}, a pointwise dissipation condition is insufficient to guarantee ISS.}


\ac{Since ISS cannot be systematically ensured through a pointwise dissipation alone, it is worth identifying the least conservative assumptions under which the conjecture holds true. Specific classes of systems have already been identified, but rely on growth restrictions on the vector field or on the LKF itself \cite{CPM17,chaillet2023growth, loko2024growth}. In \cite{KLW17}, the authors do not impose growth restrictions, but employ an ISS LKF in the so-called ``implication form'' and showed that if pointwise dissipation holds whenever the LKF dominates the input magnitude, then ISS follows. Still, this condition remains significantly more conservative than the existence of an ISS LKF with pointwise dissipation.}


In this paper, we partially solve this question by showing that, if the LKF that dissipates pointwisely can also be used to establish uniform global stability (UGS), then the ISS property holds. 
For input-free systems, our results recover and, in fact, strengthen the Krasovskii theorem for asymptotic stability \cite{Kra63}. \fw{It is an interesting side effect that the result has a unifying property, as it not only recovers the case of asymptotic stability in the input-free case. Indeed, the proof in \cite{Kra63} or also in \cite[Theorem~2.1]{HaV93} explicitly uses the fact that the LKF with pointwise dissipation in the input-free case is (with no extra assumptions) also an LKF establishing uniform global stability. For systems with inputs, this condition does not follow automatically, so we assume it. 
It is important to mention that another LLM counterexample \cite[Section 9]{BCMW26} shows that the original conjecture fails even for systems which are uniformly globally stable (in this way, this counterexample strengthens that from \cite{de2026candidate}). Hence, it is important that the UGS property is established by the same functional that dissipates pointwise.
}


Our main result actually establishes a stronger property, which we call $V$-ISS. This property is similar to ISS, but measures the behavior of the system's solutions through a given LKF $V$ rather than through the classical $\sup$ norm of the state. We believe this notion may be of interest on its own as, depending on the considered LKF, it may provide a tighter estimate of the solutions' behavior. 
In particular, we prove a $V$-ISS superposition theorem (a characterization of $V$-ISS in terms of respective attractivity and stability notions, see \cite{SoW96, MiW18b})  
and use it to derive Lyapunov-Krasovskii sufficient conditions.

The paper is organized as follows. In Section
\ref{sec:mainresult}, we recall some basics about time-delay systems, introduce $V$-ISS, and state our main result. In Section \ref{sec:discussion}, we compare this result with other approaches existing in the literature that ensure ISS based on a pointwise dissipation \cite{KLW17, loko2024growth, orlowski2022adaptive}. We also provide an academic example illustrating the added value of our approach.
In Section \ref{sec:Vstability}, we adapt some other classical stability concepts, in the same spirit as $V$-ISS, and state a superposition principle for $V$-ISS. We also provide LKF-based conditions to establish these properties and give technical observations needed for the proof of the main result in Section \ref{sec:proof}.
\lb{In Section~\ref{sec:zero-gain-input-uniformity}, we use our main result to derive a pointwise sufficient condition for uniform global asymptotic stability, extending the classical Lyapunov-Krasovskii theorem to systems with disturbances.}

A conference version of this paper, in which the main results are stated without detailed proofs, examples, or discussion, has appeared in \cite{MWC24b}.

\vspace{2mm}
\emph{Notation.} \amc{We use $\N=\{1,2,\ldots\}$.} By $\R_+:= [0,\infty)$, we denote the set of nonnegative real numbers. For $x\in\mathbb R^n$, $|x|$ denotes its Euclidean norm and $|A|$ denotes the corresponding induced matrix norm of $A\in \R^{n\times n}$.
\lb{Given metric spaces $\mathcal I$ and $\mathcal J$, $C(\mathcal I,\mathcal J)$ denotes the set of continuous maps from $\mathcal I$ to $\mathcal J$.}
``For all $t\in\mathcal I$ $a.e.$'' means for all $t\in \mathcal I\subset \mathbb R$,  except possibly on a set of zero Lebesgue measure. 
\amc{Given $\theta>0$, we set $\mathcal X:=C([-\theta,0],\mathbb R)$, and for $n\in\N$ we identify $\mathcal X^n$ with $C([-\theta,0],\R^n)$, equipped with $\|\phi\|:=\max_{s\in[-\theta,0]}|\phi(s)|$.}

$\mathcal U$ denotes the space $L^\infty_{loc}(\R_+,\R)$ of all signals $u: \R_+\to\mathbb R$ that are Lebesgue measurable and locally essentially bounded. 
Given an interval $\mathcal I\subset \R_+$ and a locally essentially bounded signal $u:\mathcal I\to\mathbb R^m$, $\|u\|:=\textrm{ess\,sup}_{t\in\mathcal I} |u(t)|$. Given $u\in\mathcal U^m$, $u_{\mathcal I}:\mathcal I\to\mathbb R^m$ denotes its restriction to the interval $\mathcal I$, in particular $\|u_{\mathcal I}\|=\textrm{ess\,sup}_{t\in\mathcal I} |u(t)|$. 
We also use the standard classes of comparison functions:
\begin{align*}
    \amc{\PD}&\amc{:=\{\alpha\in C(\R_+,\R_+) \,|\, \alpha(0)=0,\ \alpha(r)>0\text{ for }r>0\},}\\
    \K &:= \{\gamma\in C(\R_+,\R_+) \,|\, \gamma(0) = 0,\ \gamma \text{ is strictly increasing}\}, \\
    \K_\infty &:= \{\gamma\in \K \,|\, \gamma \text{ is unbounded}\}, \\
    \LL &:= \{\gamma\in C(\R_+,\R_+) \,|\, \gamma \text{ is strictly decreasing, }
    \\
    &\hspace{6.1cm}\lim_{t \to \infty}\gamma(t) = 0\}, \\
    \K\LL &:= \{\beta \in C(\R_+ \times \R_+,\R_+) \,|\, \beta(\cdot,t) \in \K \ \forall t \geq 0, \\
    &\hspace{5.5cm}\beta(r, \cdot) \in \LL\ \forall r > 0\}.
\end{align*}

\section{Preliminaries and main result}\label{sec:mainresult}

\subsection{Time-delay systems}

We consider retarded differential equations of the form
\begin{align}
\dot{x}(t)=f(x_t,u(t)),
\label{eq:time-delay}
\end{align}
where $x_t\in \mathcal X^n$, $n\in\N$, denotes the history function defined as $x_t(s):=x(t+s)$ for all $s\in[-\theta,0]$, and $\theta>0$ is the fixed maximal time-delay involved in the dynamics. 
The input $u$ is assumed to be in $\mathcal U^m$, $m\in\N$. Our standing assumption on the vector field $f$ is the following.
\begin{Ass}
\label{ass:Regularity-of-f} 
The vector field $f:\X \times \R^m \to \R^n$ is
\begin{enumerate}[label=(\roman*)]
	\item Lipschitz
          continuous in its first argument on bounded
subsets of ${\mathcal X^n}\times \R^m$, i.e.,  for
all $C>0$, there exists $L_f(C)>0$, such that for all $\phi,\varphi \in \mathcal X^n$ with $\|\phi\|\leq C$ and $\|\varphi\|\leq C$ and all $v\in\R^m$ with $|v|\leq C$,
\begin{align}
|f(\phi,v)-f(\varphi,v)| \leq L_f(C) \|\phi-\varphi\|;
\label{eq:Lipschitz}
\end{align}
	\item jointly continuous in its arguments.
\end{enumerate}
\end{Ass}

Our aim in this paper is to analyze the stability of the equilibria. For definiteness, we assume that $0$ is an equilibrium of the undisturbed system, that is: 
\begin{Ass}
\label{ass:0-is-an-equilibrium} 
Throughout $f(0,0)=0$.
\end{Ass}

By \cite[Theorem 2]{CKP23}, Assumption~\ref{ass:Regularity-of-f} guarantees that, for any initial condition $x_0 \in \X$ and any input $u\in\Uc$, there is a unique maximal solution (in the Carath{\'e}odory sense) of \eqref{eq:time-delay}, which we denote by $x(\cdot,x_0,u)$. Given $t$ in the domain of existence of this solution, the corresponding history function is denoted by $x_t(x_0,u)\in\mathcal X^n$. 
The triple $({\mathcal X^n},\Uc,\varphi)$, where $\varphi$ is the flow mapping $(x_0,u)$ and $t$ in the maximal interval of existence to $\varphi(t,x_0,u) := x_t(x_0,u)$ defines an abstract control system in the sense of \cite{MiW18b}. 
In view of Assumption~\ref{ass:Regularity-of-f}, the system \eqref{eq:time-delay} satisfies the boundedness-implies-continuation (BIC) property, i.e., every maximal solution that is bounded on its whole domain of existence is defined on $\R_+$; see \cite[Theorem 2]{CKP23}.

\subsection{Definitions}

In this section, we introduce a variant of ISS,  specifically tailored to the analysis of delay systems using Lyapunov-Krasovskii functionals. To this aim, we first recall the following definition \cite{CKP23}.

\vspace{2mm}
\begin{definition}[LKF candidate]
\label{def:Sandwiched-map}
A map $V\in C({\mathcal X^n},\R_+)$ is called a \emph{Lyapunov-Krasovskii functional candidate (LKF candidate)}, if there are $\psi_1, \psi_2 \in\Kinf$ so that the following sandwich bounds hold:
\begin{equation}
\label{eq:Sandwich}
\psi_1(|\phi(0)|) \leq V(\phi) \leq \psi_2(\|\phi\|) \quad \forall \phi \in\X.
\end{equation} 
It is said to be \emph{coercive} if, for $\psi_1$, $\psi_2$ as above, the following stronger condition is satisfied:
\begin{align*}
\psi_1(\|\phi\|) \leq V(\phi) \leq \psi_2(\|\phi\|) \quad \forall \phi \in\X.
\end{align*} 
\end{definition}
\vspace{2mm}

\begin{remark}
    We note that in \lb{\cite[Definition 8]{CKP23}}, it is required that LKF candidates are also Lipschitz continuous on bounded sets of $\mathcal {X}^n$. 
    For our results, however, we do not need this extra assumption.
\end{remark}

Next, we revisit the notion of ISS by estimating the system's properties through an associated LKF candidate, rather than through the standard norms.

\vspace{2mm}
\begin{definition}[$V$-ISS / ISS]
\label{def:viss}
Given an LKF candidate $V:\X\to\R_+$, the system \eqref{eq:time-delay} is called
\emph{$V$-input-to-state stable ($V$-ISS)} if there exist $\beta \in \KL$ and $\gamma \in \K_\infty$ 
such that, for all $ x_0 \in {\mathcal X^n}$ and all $ u\in \Uc$,
\begin{equation}
\label{eq:V-iss_sum}
V(x_t(x_0,u) ) \leq \beta( V(x_0),t) + \gamma( \|u\|),\quad \forall t\geq 0.
\end{equation}
System \eqref{eq:time-delay} is called \emph{input-to-state stable (ISS)}, if it is $V$-ISS for $V(\phi) = \|\phi\|$.
\end{definition}

\begin{remark}
In the above \lb{Definition \ref{def:viss}}, we tacitly assume that all solutions are well-defined on $\R_+$. Similar conventions are taken for other definitions in this work. If \eqref{eq:V-iss_sum} were to hold only on the domain of \lb{existence} of $x(\cdot,x_0,u)$, then \eqref{eq:V-iss_sum} would guarantee that the solution \amc{is} bounded on this interval, and the BIC property of \eqref{eq:time-delay} guaranteed by our regularity Assumption~\ref{ass:Regularity-of-f} would imply that the domain of \lb{existence} is, in fact, $\R_+$, and thus \eqref{eq:V-iss_sum} is valid on $\R_+$.
\end{remark}

\amc{
Note that if $V$ and $W$ are equivalent in a certain sense, then $V$-ISS is also equivalent to $W$-ISS.
\begin{proposition}
\label{prop:V-ISS-Kinf-invariance}
Let $V,W:\mathcal X^n\to\R_+$ be  LKF candidates and let there exist $\psi_1,\psi_2 \in \Kinf$ such that
\[
\psi_1(V(x))\leq W(x)\leq \psi_2(V(x)),\quad x\in \X.
\]
Then \eqref{eq:time-delay} is $V$-ISS if and only if \eqref{eq:time-delay} is $W$-ISS. 

In particular, if $V$ is coercive, then $V$-ISS is equivalent to ISS.
\end{proposition}

\begin{proof}
If \eqref{eq:time-delay} is $V$-ISS, then there are
$\beta\in\KL$ and $\gamma\in\Kinf$ such that, for all
$x_0\in\mathcal X^n$ and all $u\in\Uc$, we have
\eqref{eq:V-iss_sum}. Using the comparison bounds between
$V$ and $W$ and the inequality
$\psi_2(a+b)\leq\psi_2(2a)+\psi_2(2b)$, $a,b\geq0$,
we obtain, for all $t\geq0$,
\begin{align*}
W(x_t(x_0,u))
&\leq \psi_2\bigl(V(x_t(x_0,u))\bigr)\\
&\leq \psi_2\bigl(\beta(V(x_0),t)+\gamma(\|u\|)\bigr)\\
&\leq \psi_2\bigl(2\beta(\psi_1^{-1}(W(x_0)),t)\bigr)
      +\psi_2\bigl(2\gamma(\|u\|)\bigr).
\end{align*}
Thus, \eqref{eq:time-delay} is $W$-ISS with the gain
$r\mapsto\psi_2\bigl(2\gamma(r)\bigr)\in\Kinf$ and the transient
$(r,t)\mapsto
\psi_2\bigl(2\beta(\psi_1^{-1}(r),t)\bigr)\in\KL$.

The converse follows by the same argument, interchanging
$V$ and $W$.
\end{proof}

}

\amc{If $V$ is non-coercive,} $V$-ISS provides a tighter estimate on the solutions' norm. For instance, consider the following widely-used class of quadratic LKF candidates
\begin{align*}
    V(\phi):=\phi(0)^\top P \phi(0)+\int_{-\theta}^0 \phi(\tau)^\top Q \phi(\tau)d\tau,
\end{align*}
where $P,Q\in\R^{n\times n}$ denote symmetric, positive definite matrices. For such LKF candidates, $V$-ISS ensures an upper bound on the solution's norm in terms of $V(x_0)$, and, since $P,Q$ are positive definite, in terms of $|x_0(0)|^2+\int_{-\theta}^0|x_0(\tau)|^2d\tau$, whereas the classical ISS would upper-bound them in terms of $\|x_0\|$, which may be significantly larger for some particular initial states. 


\ac{Note also that $V$-ISS is nothing else than \emph{state-independent input-to-output stability} \cite[Definition 2.1]{kankanamalage2019remarks} with output $y=V(\phi)$. 
Input-to-output stability has received a wide attention in a delay-free context \cite{SoW99,SoW00}, with more recent extensions to time-delay systems \cite{KLW17,kankanamalage2019remarks,KaJ11b,CKP23}.}

The following statement clarifies the properties induced by $V$-ISS and relates it to the classical definition of ISS.

\vspace{2mm}
\begin{proposition}[$V$-ISS $\Rightarrow$ ISS]
\label{prop:V-ISS_and_ISS} 
Given an LKF candidate $V:\X\to\R_+$, consider the following statements:
\begin{enumerate}[label=\roman*)]
	\item \label{itm:V-ISS_and_ISS-1} System \eqref{eq:time-delay} is $V$-ISS.

	\item \label{itm:V-ISS_and_ISS-2} There exist $\beta \in\KL$ and $\gamma \in\Kinf$ such that, for all $x_0 \in {\mathcal X^n}$ and all $u\in\Uc$, the flow of \eqref{eq:time-delay} satisfies
\begin {equation}
\label{eq:v-norm-iss_sum}
V(x_t(x_0,u)) \leq \beta(\|x_0\|,t) + \gamma( \|u\|),\quad \forall t\geq 0.
\end{equation}
		
	\item \label{itm:V-ISS_and_ISS-3} There exist $\beta \in\KL$ and $\gamma \in\Kinf$ such that, for all $x_0 \in {\mathcal X^n}$ and all $u\in\Uc$, the flow of \eqref{eq:time-delay} satisfies
\begin{align}
\label{Gleichung_Verzoeg_ISS}
|x(t,x_0,u)|\leq \beta(\|x_0\|,t) + \gamma(\left\|u\right\|),\quad \forall t\geq 0.
\end{align}
	
	\item \label{itm:V-ISS_and_ISS-4} System \eqref{eq:time-delay} is ISS.
\end{enumerate}
Then the following relations hold:
\begin{center}
	i) $\Rightarrow$ ii) \ $\Iff$ \ iii) \ $\Iff$ \ iv).
\end{center}
If $V$ is coercive, then all four statements are equivalent.
\end{proposition}

\begin{proof}
i) $\Rightarrow$ ii) $\Rightarrow$ iii) follows easily from the fact that by definition $V$ satisfies a sandwich condition as in \eqref{eq:Sandwich}. 
\amc{
The fact that iv) $\Rightarrow$ ii) is straightforward by noticing that by \eqref{eq:Sandwich} we have $V(\phi) \leq \psi_2(\|\phi\|)$ for a certain $\psi_2\in\Kinf$ and for all $\phi \in \X$. 
}

Let us prove that iii) $\Rightarrow$ iv). \ac{This was already established in e.g. \cite[Proposition 1.4.2]{Mir12}, but we report the proof for the sake of completeness.} Notice that iii) guarantees that, for all $t\geq \theta$,
\begin{align}
	\|x_t(x_0,u)\| &= \max_{\tau \in [-\theta,0]}|x(t+\tau,x_0,u)| \nonumber\\
	&\leq \max_{\tau \in [-\theta,0]} \beta(\|x_0\|,t + \tau) +\gamma(\left\|u\right\|)\nonumber\\
&=\beta(\|x_0\|,t-\theta)+\gamma(\left\|u\right\|). \label{eq-6}
\end{align}
On the other hand, for $t \in [0,\theta]$, it holds that 
\begin{align*}
\|x_t(x_0,u)\| &= \max\left\{ \max_{s\in[t-\theta,0] }|x(s,x_0,u)|, \max_{ s\in [0,t]}|x(s,x_0,u)|\right\} \\
&\leq \max\left\{ \|x_0\|, \beta(\|x_0\|,0) + \gamma(\left\|u\right\|) \right\}.
\end{align*}
Since $\beta(r,0)\geq r$ for all $r\geq 0$ (as can be seen from \eqref{Gleichung_Verzoeg_ISS} with $u \equiv 0$), we obtain that
\begin{align}\label{eq-7}
\|x_t(x_0,u)\| \leq \beta(\|x_0\|,0) + \gamma(\left\|u\right\|),\quad \forall t \in [0,\theta].
\end{align}
Combining \eqref{eq-6} and \eqref{eq-7}, we conclude that
\begin{align*}
    \|x_t(x_0,u)\| \leq \tilde\beta(\|x_0\|,t)+\gamma(\left\|u\right\|),\quad \forall t\geq 0,
\end{align*}
where $\tilde\beta$ is the function defined for all $r,t\geq 0$ as 
\[
\tilde{\beta}(r,t)=
\left\{ 
\begin{array}{ll}
(\theta -t)r + \beta(r,0),&\textrm { if }t \in [0, \theta],\\
\beta(r,t - \theta),&\textrm{ if } t>\theta. 
\end{array} \right.
\]
The conclusion follows by noticing that $\tilde \beta\in\KL$.

\amc{Finally, in view of Proposition~\ref{prop:V-ISS-Kinf-invariance}, ii) $\Rightarrow$ i)  when $V$ is coercive. 
}
\end{proof}

\amc{
\begin{remark}[\ac{ISS $\nRightarrow$ $V$-ISS}]
\label{rem:ISS-but-not-V-ISS}
According to \cite[Example 6.2]{BCMW26},     
the implication (i) $\Rightarrow$ (ii) shown in Proposition~
\ref{prop:V-ISS_and_ISS}, cannot be reversed, i.e., ISS does not necessarily imply $V$-ISS, even if $V$ is an ISS LKF in pointwise implication form.
\end{remark}
}

To analyze the $V$-ISS property using ISS Lyapunov-Krasovskii functionals, we use the following notions of ISS LKF. They all rely on the \emph{upper right-hand Dini derivative} of the map $V$ along the solutions of system \eqref{eq:time-delay}, defined for all $\phi\in\X$ and $u\in\mathcal U^m$ as
\begin{equation}
\label{eq:Dinidef}
  \dot{V}_u(\phi) := \limsup_{h\to 0^+}\frac{V(x_h(\phi,u))-V(\phi)}{h}.  
\end{equation}

\vspace{2mm}
\begin{definition}[Pointwise/LKF-wise ISS LKF]
\label{def:ISS_LKF-implication}
For system \eqref{eq:time-delay}, an LKF \amc{candidate} $V:{\mathcal X^n} \to \R_+$ is called:
\begin{enumerate}[label=(\roman*)]
\item\label{itm:lkf1}  an \emph{\ac{ISS LKF in LKF-wise implication form}} if there exist $\chi \in \Kinf$ and $\alpha \in\PD$ 
such that, for all $\phi \in {\mathcal X^n}$ and all $u\in\Uc$,
\begin{equation}
\label{eq:ISS-LK-implication-V}
V(\phi)\geq \chi(\|u\|) \qrq \dot{V}_u(\phi) \leq -\alpha(V(\phi)).
\end{equation}

\item\label{itm:lkf2} an \emph{\ac{ISS LKF in pointwise implication form}} if there exist $\chi \in \Kinf$ and $\alpha \in\PD$ such that, for all $\phi \in {\mathcal X^n}$ and all $u\in\Uc$,
\begin{equation}
\label{eq:ISS-LK-implication-relaxed}
|\phi(0)| \geq \chi(\|u\|) \qrq \dot{V}_u(\phi) \leq -\alpha(|\phi(0)|).
\end{equation}
\item an \emph{ISS LKF  in \ac{LKF-wise dissipative form}} if there exist $\alpha,\chi \in \Kinf$ such that, for all $\phi \in {\mathcal X^n}$ and all $u\in\Uc$,
\begin{equation}
\label{eq:ISS-LK-sum-LKFwise}
\dot{V}_u(\phi) \leq -\alpha(V(\phi))+\chi(\|u\|).
\end{equation}

\item  an \emph{ISS LKF  in \ac{pointwise dissipative form}} if there exist $\alpha,\chi \in \Kinf$ such that, for all $\phi \in {\mathcal X^n}$ and all $u\in\Uc$,
\begin{equation}
\label{eq:ISS-LK-sum-relaxed}
\dot{V}_u(\phi) \leq -\alpha(|\phi(0)|)+\chi(\|u\|).
\end{equation}
\end{enumerate}
In all cases, the function $\chi$ is called the gain related to $V$.
\end{definition}
\vspace{2mm}

\begin{remark}[\ac{Dini vs. Driver}]
\label{rem:DinivsDriver}
    In Definition~\ref{def:Sandwiched-map} we imposed little regularity on $V$ -- in part because for many arguments only continuity is required.
    \lb{
    For example, it is sufficient to have $V$ continuous to ensure that, if $\dot V_u(\phi) \leq 0$ for all $\phi\in \X$ and all $u\in \Uc$, then $t\mapsto V_u(x_t(x_0, u))$ is non-increasing (see e.g.\cite{hagood2006recovering}).
    }
    On the other hand, this makes the verification of the conditions appearing in Definition~\ref{def:ISS_LKF-implication} tedious, because the expression in \eqref{eq:Dinidef}
    \lb{involves the full input signal $u\in \Uc$}.
    Frequently, it is assumed that $V$ is locally Lipschitz continuous; see, e.g., \cite{CKP23}. In this case, the Driver derivative \cite{driver1962existence,CKP23} is useful. Given $\phi \in \X$, \lb{and a vector} $\muval \in \R^m$, define pseudotrajectories for $h\in[0,\theta)$ by
    \begin{equation}
    \phi_{h,\muval}(\tau) := \left\{ \begin{matrix}
        \phi(\tau + h) & \tau \in [-\theta,- h], \\
        \phi(0) + \amc{(\tau+h)} f(\phi,\muval) & \tau \in [-h,0].
    \end{matrix}\right.
\end{equation}
The Driver derivative of $V$ at $\phi$ in direction $f(\phi,\muval)$ is then
\begin{equation}
\label{eq:Driverdef}
    D^+V(\phi,\muval) := \limsup_{h\to 0^+} \frac{1}{h}\left( V(\phi_{h,\muval}) - V(\phi) \right).
\end{equation}
If $V$ is locally Lipschitz continuous, then for any input $u\in \Uc$ the Dini derivative and the pointwise Driver derivative coincide almost everywhere along a trajectory, \cite[Theorem 2]{Pepe07-correct}. In addition, 
\fw{following \cite{pepe2007problem}, solutions corresponding to initial conditions $\phi \in C^1([-\theta,0],\R^n)$ and $u\in \Uc$ are (locally) absolutely continuous and in this case, for locally Lipschitz continuous $V$, the composition $V\circ x_t$ is absolutely continuous. 
As $C^1([-\theta,0],\R^n)$ is norm-dense in $\X$,}
it is not hard to see that\fw{, in the situation we have described, it is sufficient to formulate the conditions of Definition~\ref{def:ISS_LKF-implication} in terms of Driver derivatives, e.g. instead of \eqref{eq:ISS-LK-implication-V} a sufficient decay condition be expressed as}
\begin{equation}
\label{eq:KLW-condition}
  V(\phi)\geq \chi(|\muval|) \qrq D^+{V}(\phi,\muval) \leq -\alpha(V(\phi)),
  \tag{\ref{eq:ISS-LK-implication-V}'}
\end{equation}
and similarly for \eqref{eq:ISS-LK-implication-relaxed}-\eqref{eq:ISS-LK-sum-relaxed}.
\end{remark}

In the dissipative form of the ISS LKF, we have to require the decay rate $\alpha$ to be a $\Kinf$-function, \ac{in line with classical ISS results for finite-dimensional systems \cite{SoW95,Mir23}.} At the same time, for ISS LKFs in implication form, it suffices that $\alpha$ is merely a $\PD$-function. As the following proposition shows (motivated by \cite[Remark 4.1]{LSW96}, \cite[Proposition 2.17]{Mir23}), by a suitable scaling of an ISS Lyapunov functional (in implication form), we can always obtain a $\Kinf$ decay rate. Recall that a \emph{nonlinear scaling} is a function $\mu\in\Kinf$, continuously differentiable on $(0,+\infty)$, satisfying  $\mu'(s)>0$ for all $s>0$ and such that $\lim_{s\to 0}\mu'(s)$ exists and is finite (that is, it belongs to $\R_+$).

\begin{proposition}[\ac{From $\PD$ to $\mathcal K_\infty$ dissipation rate}]
\label{prop:Kinf-decay-rate-always-suffices-pointwise-dissipation} 
Let $V \in C({\mathcal X^n},\R_+)$ be an ISS LKF in pointwise implication form with a decay rate $\alpha\in\PD$. Then there is a nonlinear scaling $\xi \in\Kinf$, which belongs to $C^1(\R_+,\R_+)$, is infinitely differentiable on $(0,+\infty)$, and such that $\xi\circ V(\cdot)$ is an ISS LKF in pointwise implication form with a decay rate belonging to $\Kinf$.
\end{proposition}

\begin{proof}
Let $V \in C({\mathcal X^n},\R_+)$ be an ISS LKF in pointwise implication form with sandwich bounds $\psi_1,\psi_2\in\Kinf$, decay rate $\alpha\in\PD$ and gain $\chi\in\Kinf$.
\amc{By \cite[Lemma IV.1]{ASW00},} there are $\omega \in\Kinf$ and $\sigma \in \LL$ such that 
\begin{align*}
\alpha(r) \geq \omega(r)\sigma(r),\quad \forall r\geq 0.
\end{align*}
Without loss of generality, we can assume that $\sigma$ and $\psi_1$ are infinitely differentiable on $(0,+\infty)$ (otherwise take $\frac{1}{2}\sigma$, $\frac{1}{2}\psi_1$ and smooth them\amc{, see, e.g., \cite[Proposition A.16]{Mir23}}). From \eqref{eq:ISS-LK-implication-relaxed}, it follows that for all  $\phi \in {\mathcal X^n}$ and $u \in \Uc$ 
\begin{equation*}
\label{eq:ISS-LK-implication-relaxed-1}
|\phi(0)| \geq \chi(\|u\|) \qrq \dot{V}_u(\phi) \leq -\omega(|\phi(0)|)\sigma(|\phi(0)|).
\end{equation*}
As $\sigma \in \LL$, $-\sigma$ is an increasing function, and by the sandwich inequality 
\eqref{eq:Sandwich}, we have that
\begin{equation*}
|\phi(0)| \geq \chi(\|u\|) \qrq \dot{V}_u(\phi) \leq -\omega(|\phi(0)|)\sigma \circ \psi_1^{-1}(V(\phi)).
\end{equation*}
Thus, for all $\phi \in{\mathcal X^n}$ and all $u \in \Uc$, we have
\begin{align}
|\phi(0)| \geq \chi(\|u\|) \qrq \frac{1}{\sigma\circ \psi_1^{-1}(V(\phi))}\dot{V}_u(\phi) \leq -\omega(|\phi(0)|).
\label{ImplicationIneq_cont-aux2}
\end{align}
Define $\xi \in C^1(\R_+,\R_+)$ by 
\begin{eqnarray}
\xi(r):=\int_0^r\frac{1}{\sigma\circ \psi_1^{-1}(s)}ds, \quad r \geq 0.
\label{eq:scaling-LF-def}
\end{eqnarray}
\amc{
As $\sigma(r)>0$ for all $r\ge0$, the integrand is continuous at zero}\fw{. By assumption, it is} \amc{smooth on $(0,+\infty)$. Hence $\xi\in C^1(\R_+,\R_+)$ and it is infinitely differentiable on $(0,+\infty)$. \ac{Moreover, $\xi^{\prime}(0)=1/\sigma(0)>0$ and $\xi'(r)=1/\sigma\circ \psi_1^{-1}(s)>0$ for all $s>0$}
, so $\xi\in\Kinf$ is a nonlinear scaling.


Define $W:=\xi\circ V$.
As $\xi'$ is positive, 
the chain rule for Dini derivatives \cite[Lemma A.1]{CMW26c} yields that 
$\dot{W}_u(\phi)=\frac{1}{\sigma \circ\psi_1^{-1}(V(\phi))}\dot{V}_u(\phi)$. Thus \eqref{ImplicationIneq_cont-aux2} transforms to
\begin{equation}
|\phi(0)| \geq \chi(\|u\|)
\quad \Rightarrow \quad \dot{W}_u(\phi) \leq - \omega(|\phi(0)|),
\label{ImplicationIneq_cont-aux3}
\end{equation}
with $\omega \in\Kinf$. 
As
\begin{equation}
\label{LyapFunk_1Eig-aux-W}
\xi\circ \psi_1(|\phi(0)|) \leq W(\phi) \leq \xi\circ \psi_2(\|\phi\|), \quad  \phi \in {\mathcal X^n},
\end{equation}
}
we see that $W$ is an ISS LKF in
pointwise implication form with a $\Kinf$ decay rate. 
\end{proof} 

\begin{remark}
A counterpart of Proposition~\ref{prop:Kinf-decay-rate-always-suffices-pointwise-dissipation} can also be shown  \amc{for ISS LKFs in LKF-wise implication form}.
\end{remark}

The following result states that any ISS LKF  in pointwise \sumform{} is also an ISS LKF in pointwise implication form.

\vspace{2mm}
\begin{proposition}(\textbf{\Sumform{}} $\Rightarrow$ \textbf{implication form})
\label{prop-sum-imp}
For system \eqref{eq:time-delay}, if $V$ is an ISS LKF in pointwise \sumform{},  then it is also an ISS LKF in pointwise implication form.
\end{proposition}
\vspace{2mm}

\begin{proof}
If $V$ is an ISS LKF in pointwise \sumform{}, then \eqref{eq:ISS-LK-sum-relaxed} holds with some $\alpha,\chi\in\Kinf$. Thus,
\begin{align*}
|\phi(0)| \geq \alpha^{-1}\circ 2\chi(\|u\|) \qrq \dot{V}_u(\phi) \leq -\frac{1}{2}\alpha(|\phi(0)|),
\end{align*}
and the claim follows.
\end{proof}
\vspace{2mm}

\ac{The following notions are central to our main result.}

\vspace{2mm}
\begin{definition}[UGS, $V$-UGS]
\label{def:VUGS}
Given an LKF candidate $V:\ \mathcal X^n\to\R_+$, the system \eqref{eq:time-delay} is  called \emph{$V$-uniformly globally stable ($V$-UGS)}, if there exist $ \sigma\in\mathcal K_\infty$ and $\gamma
          \in \ac{\Kinf\cup\{0\}}$ such that for all $ x_0 \in \mathcal X^n$ and all $u
          \in \Uc$, its solution satisfies
\begin{equation}
\label{eq:V-GSAbschaetzung}
V(x_t(x_0,u)) \leq \sigma(V(x_0)) + \gamma(\|u\|) \quad \forall t \geq 0.
\end{equation}          
The system is said to be \emph{uniformly globally stable (UGS)} if it is $V$-UGS with $V(\cdot)=\|\cdot\|$. \ac{The function $\gamma$ is referred to as a \emph{$V$-UGS or UGS gain}.}
\end{definition}
\vspace{2mm}

\ac{UGS is a direct extension of its finite-dimensional counterpart (see the definition of GS in \cite{SoW96}) to time-delay systems \cite{MiW17e,MIWICHABR24super}. It ensures uniform boundedness of solutions and small response to inputs of sufficiently small magnitude when starting sufficiently close to the origin. $V$-UGS allows a potentially tighter estimate by quantifying the solutions through the LKF $V$ rather than its $L_\infty$ norm. Exploiting the sandwich bound on the LKF $V$ and similar reasoning as in Proposition \ref{prop:V-ISS_and_ISS}, we have the following.}

\vspace{2mm}
\begin{proposition}[$V$-UGS $\Rightarrow$ UGS]
\label{prop:V-UGS_and_UGS} 
Given an LKF candidate $V:\ \mathcal X^n\to\R_+$, if \eqref{eq:time-delay} is $V$-UGS, then it is UGS. 
\end{proposition}

\ac{It turns out that $V$-UGS can be ensured using the following relaxation of the concept of ISS LKF in implication form, in which the negative definite dissipation is replaced by a non-positivity requirement.} 

\vspace{2mm}
\begin{definition}[UGS LKF]
\label{def:noncoercive_UGS_LF}
An LKF candidate $V:{\mathcal X^n} \to \R_+$ is called a \emph{UGS LKF} for \eqref{eq:time-delay} if there exists $\chi \in\Kinf$ such that, for all $\phi \in {\mathcal X^n}$ and all $u\in\Uc$,
\begin{equation}
\label{DissipationIneq_UGS_LK_functionals}
V(\phi) \geq \chi(\|u\|) \qrq \dot{V}_u(\phi) \leq 0.
\end{equation}
\fw{The function $\chi$ is also called the UGS LKF gain of $V$.}
\end{definition}
\vspace{2mm}

\fw{\begin{remark}[Dini Vs. Driver] 
Using the same arguments as in Remark~\ref{rem:DinivsDriver}, again \eqref{DissipationIneq_UGS_LK_functionals} may be expressed in terms of the Driver derivative provided that $V$ is locally Lipschitz continuous.
\end{remark}}


\ac{The next result states that the existence of a UGS LKF $V$ indeed guarantees $V$-UGS. The proof is very similar to that of  \cite[Corollary 4.2]{kankanamalage2019remarks} in which the UGS property is referred to as UBIBS. The details are left to the reader.}

\vspace{2mm}
\begin{proposition}[LKF condition for $V$-UGS]
\label{prop:UGS_LK_theorem} 
If  \eqref{eq:time-delay} admits a UGS LKF $V$, then \eqref{eq:time-delay} is $V$-UGS (and thus UGS).
\end{proposition}
\vspace{2mm}

\subsection{Main result}

In \cite{CPM17}, it has been conjectured that the existence of an ISS LKF in pointwise \sumform{} is enough to ensure ISS. \ac{Although this has been proven for various classes of time-delay systems \cite{CPM17, loko2024growth,chaillet2023growth}, the recent counterexample \cite{de2026candidate} suggests that this conjecture is wrong in general.}
Our main result \ac{extends the class of systems for which ISS can be established through a pointwise dissipation by showing that} ISS (and, actually, $V$-ISS) indeed holds if $V$ is also a UGS LKF.

\vspace{2mm}
\begin{theorem}[ISS with pointwise dissipation]
\label{thm:Krasovskii-Theorem} 
Let Assumptions~\ref{ass:Regularity-of-f} and \ref{ass:0-is-an-equilibrium} hold.  
If there exists an LKF candidate $V:\mathcal X^n\to\R_+$ which is simultaneously an ISS LKF with pointwise dissipation (in either implication or \sumform) and a UGS LKF for \eqref{eq:time-delay}, then \eqref{eq:time-delay} is $V$-ISS and, in particular, ISS.
\end{theorem}
\vspace{2mm}

\ac{Note that, although this conjecture was formulated in a dissipative form, our main result also covers ISS LKF in implication form and guarantees a stronger $V$-ISS property.}

The proof of this result requires the introduction of further notions related to $V$-stability and corresponding LKF tools. It is therefore postponed to Section~\ref{sec:proof}.

\ac{Since the existence of a UGS LKF $V$ guarantees $V$-UGS, a natural question is whether $V$-UGS would be enough to guarantee ISS if the LKF $V$ is an ISS LKF with pointwise dissipation. An LLM-generated example \cite[Theorem 9.1]{BCMW26} provides a negative answer to that question and thus highlights the key role played by the fact that, in Theorem \ref{thm:Krasovskii-Theorem}, the same LKF candidate $V$ needs to be both an ISS LKF with pointwise dissipation and a UGS LKF.}

\amc{
\begin{remark}[\ac{Pointwise ISS LKF $V$ \& $V$-UGS $\nRightarrow$ ISS}]
There exists a time-delay system satisfying Assumptions~\ref{ass:Regularity-of-f} and \ref{ass:0-is-an-equilibrium} with the following properties (see \cite[Theorem 9.1]{BCMW26}):
\begin{itemize}
    \item it admits an ISS LKF $V:\mathcal X^n\to\R_+$ with pointwise dissipation (in \sumform)
    \item it is $V$-UGS \ac{(hence, UGS)}
    \item it is not ISS.
\end{itemize}
\end{remark}
}








\section{Discussion}\label{sec:discussion}



We next discuss the novelty of Theorem \ref{thm:Krasovskii-Theorem} by comparing it to other approaches in the literature. 

\subsection{Relations to existing results}
\label{subsec:relations}

A sufficient condition for ISS in terms of pointwise dissipation has been obtained in \cite{KLW17}. In the terminology of Definition~\ref{def:ISS_LKF-implication}, it can be expressed as
\begin{align}\label{eq-12}
V(\phi)\geq \chi(\|u\|)\quad\Rightarrow\quad \dot V_u(\phi)\leq -\alpha(|\phi(0)|).
\end{align}

\amc{This condition lies halfway between \eqref{eq:ISS-LK-implication-V} and \eqref{eq:ISS-LK-implication-relaxed}, in that it requires a pointwise dissipation, but only when the LKF {itself} qualitatively dominates the input norm.} It has been shown in \cite[Theorem 2]{KLW17} that \eqref{eq-12} is sufficient to ensure ISS. This result can be seen as a corollary of Theorem \ref{thm:Krasovskii-Theorem} since \eqref{eq-12} implies that $V$ is both an ISS LKF in pointwise implication form and a UGS LKF. Our result therefore strengthens \cite[Theorem 2]{KLW17} in three different ways:
$i)$ our requirements on $V$ are weaker than those in \cite[Theorem 2]{KLW17}, as $V$ is requested to decay only when \amc{$|\phi(0)| \geq \chi(\|u\|)$}, $ii)$ Theorem~\ref{thm:Krasovskii-Theorem} does not merely ensure ISS but also $V$-ISS, which is a potentially stronger property, and $iii)$ our requirements on the vector field $f$ (Assumption~\ref{ass:Regularity-of-f}) are weaker than those in \cite[Theorem 2]{KLW17}: we do not assume Lipschitz continuity of $f$ with respect to its second argument (the input $u$), which was important in \cite{KLW17}.

In a different direction, \cite{CPM17,chaillet2023growth, loko2024growth} assume growth restrictions on the LKF candidate or the vector field to establish ISS based on pointwise dissipation. In Theorem \ref{thm:Krasovskii-Theorem}, no such growth constraints are needed, at the price of a UGS requirement (through the same LKF $V$). Among these growth restriction approaches, some rely on the so-called ``exponential trick'' studied in \cite{orlowski2022adaptive} and \cite[Section~4.3]{thesisLoko2025}. This method considers an LKF candidate of the form
\begin{equation}
\label{eq:Loko-LKF-form}
    W(\phi) = w_1(\phi(0)) + \int_{-\theta}^0 w_2(\phi(s)) ds.
\end{equation}
Under suitable conditions on $W$ and the system \eqref{eq:time-delay}, it was shown that for certain $c,\kappa>0$ the modified LKF
\begin{equation}
    W_{c,\kappa}(\phi) := \kappa w_1(\phi(0)) + \int_{-\theta}^0 e^{cs}w_2(\phi(s)) ds
\end{equation}
\ac{is an ISS} LKF with LKF-wise dissipation (hence, the system is ISS), if a pointwise dissipation holds for $W$. An example in \cite{loko2024growth} shows that this is by no means a universal method, \ac{but works only for specific classes of systems.}

Another method in this line was proposed in \cite{loko2024growth} \ac{by relying on \cite[Lemma 6.7]{KaJ11b}}. 
\ac{It considers a pointwise decay estimate in the form} 
\begin{equation}\label{eq:epiphane}
    \dot V_{u}(\phi) \leq - \alpha(Q(\phi(0))) + \gamma(\|u\|),
\end{equation}
where $\alpha, \gamma \in \Kinf$ and $Q:\R^n\to \R_+$ is $C^1$, positive definite, and radially unbounded. \cite[Theorem~1]{loko2024growth} states that if there exists $\sigma \in \Kinf$ such that, for all $\phi\in  \X$ and all $v\in\mathbb R^m$, 
\begin{equation}
\label{eq:epiphane2}
    \nabla Q(\phi(0)) f(\phi, v) \leq \sigma\left( \max_{\tau \in [-\theta,0]} Q(\phi(\tau)) \right) + \gamma(|v|),
\end{equation}
\ac{then the system is ISS provided that the following growth restriction holds:}
\begin{equation}
\label{eq:epiphane3}
    \ac{\liminf}_{s\to \infty} \frac{\alpha(s)}{\sigma(se^{\theta})} >0.
\end{equation}
In \cite[Proposition 2]{loko2024growth}, an example of a delay system is presented, for which one can construct the functions $V$ and $Q$ as above to conclude ISS via \cite[Theorem~1]{loko2024growth}, \ac{although the exponential trick fails}. It was also shown that the constructed functional $V$ does not satisfy the assumptions of Theorem~\ref{thm:Krasovskii-Theorem}, \ac{thus showing that the approach proposed here complements but does not encompass all available results on ISS through pointwise dissipation}. 

We note that even though the author of \cite[Example 4.11]{thesisLoko2025} claims that their particular example shows that all systems covered by Theorem~\ref{thm:Krasovskii-Theorem} can be treated using an approach as described in \eqref{eq:epiphane}--\eqref{eq:epiphane3}, no argument is provided for why this should be the case.

\subsection{Example}

Below, we exhibit an example of an ISS system for which Theorem~\ref{thm:Krasovskii-Theorem} can be used to prove ISS, while none of the approaches described in Section~\ref{subsec:relations} can be applied to the family of quadratic LKF candidates under consideration.

Let $\psi: \R \to \R$ denote any smooth function such that $\psi(s) = 0$ for all $s\leq 0$, $\psi(s) > 0$ for all $s > 0$, and $\psi(s)s\to0$ as $s\to +\infty$, and consider 
the system
\begin{equation}
\label{eq:examplesys}
\begin{aligned}
     \dot{x}(t) & = - x(t) + x(t-1)  - \psi(x^2(t)-u^2(t))x(t),
\end{aligned}
\end{equation}
\ac{with initial state $x_0\in C([-1,0],\R)=:\mathcal X$. Then we have the following.}

\begin{proposition}
\label{prop:example-summary}
    \fw{Consider system \eqref{eq:examplesys}.}
    For all constants $c\geq0$ and $\kappa >0$, the functional defined by
    \begin{equation}
    \label{eq:Vck-def}
        V^{c,\kappa}(\phi) := \kappa\phi^2 (0)+ \int_{-1}^0 e^{cs} \phi^2(s) ds, \quad \phi \in \mathcal{X},
    \end{equation}
    is an LKF candidate that is Lipschitz continuous on bounded sets. It holds that
    \begin{enumerate}[label=(\roman*)]
    \item\label{itm:positive_result} the system  \eqref{eq:examplesys} is $V^{0,1}$-UGS with zero gain and 
    $V^{0,1}$-ISS;

    \item\label{itm:iss} the system \eqref{eq:examplesys} is ISS;

    \item for every $(c,\kappa)\neq (0,1)$, there exists a $\phi \in \mathcal{X}$ such that $|\phi(0)|>0$ and $\dot V^{c,\kappa}_{u\equiv 0}(\delta\phi) >0$ for all $\delta \in (0,1]$. In particular, $V^{c,\kappa}$ does not certify local stability with zero input 
    \ac{and it is not an ISS LKF with pointwise dissipation, in either implication form \eqref{eq:ISS-LK-implication-relaxed} or \sumform{} \eqref{eq:ISS-LK-sum-relaxed}.}

    \item \ac{there is no choice of $c\geq 0$ and $\kappa>0$ such that $V:=V^{c,\kappa}$ is an ISS LKF with LKF-wise dissipation in either implication form \eqref{eq:ISS-LK-implication-V} or \sumform{} \eqref{eq:ISS-LK-sum-LKFwise}.}

    \item\label{item:wang} there is no choice of $c\geq0$, $\kappa>0$ for which there exist $\alpha,\chi\in \Kinf$ such that \ac{\eqref{eq-12}} is satisfied for $V:=V^{c,\kappa}$.

    \item\label{itm:loko} there is no choice of $c\geq0$, $\kappa>0$ for which there exist $\alpha, \gamma \in \Kinf$, and a \amc{$Q:\R\to \R_+$} that is $C^1$, positive definite and radially unbounded, such that 
    \eqref{eq:epiphane} is satisfied for $V:=V^{c,\kappa}$.
    
    \end{enumerate}
\end{proposition}

 Proposition~\ref{prop:example-summary} highlights the following concerning the analysis methods mentioned in Section \ref{subsec:relations}.

First, for this example and for the proposed class of LKF candidates, the results of \cite{Kan17} are not applicable. \ac{In view of $(iii)$}, the only possible candidate would be $V^{0,1}$. \ac{However $(v)$ shows that, even in this case,  \eqref{eq-12} cannot be satisfied.}
        
\ac{As far as the exponential trick is concerned, $(iii)$ shows that among the family of LKFs \eqref{eq:Vck-def}, the only one that may admit a pointwise dissipation is $V^{0,1}$. All other LKF candidates $V^{c,\kappa}$ are precisely the ones resulting from the exponential trick considered in \cite{orlowski2022adaptive} and \cite[Section~4.3]{thesisLoko2025}. Since, by $(iv)$, none of them admit an LKF-wise dissipation, this shows that ISS cannot be established using these techniques either.} 

Finally, \amc{the method \ac{developed in \cite{loko2024growth} and recalled in} \eqref{eq:epiphane}--\eqref{eq:epiphane3} cannot be applied to system \eqref{eq:examplesys} using any $V^{c,\kappa}$ \ac{since the basic requirement  \eqref{eq:epiphane} is not met, as shown by $(vi)$}. On the other hand, the example discussed in \cite{thesisLoko2025} uses an LKF that does not meet the assumptions of Theorem \ref{thm:Krasovskii-Theorem}. These comparisons concern the specific class of considered LKF candidates and do not establish inclusion or non-inclusion of the classes of systems certifiable with arbitrary LKFs.} 


\vspace{2mm}
\begin{proof} (of Proposition \ref{prop:example-summary})
    It is easy to see that for every $c\geq0$, $\kappa>0$, the conditions of Definition~\ref{def:Sandwiched-map} are satisfied, meaning that $V^{c,\kappa}$ is an LKF candidate. In view of \cite[Example 1]{CKP23}, $V^{c,\kappa}$ is Lipschitz continuous on bounded sets. Following Remark~\ref{rem:DinivsDriver}, it is sufficient to consider the Driver derivative in the remainder of the proof. This has the added benefit that our considerations are directly comparable to the methods discussed in Section~\ref{subsec:relations}.
    
    \underline{Proof of $(i)$.} Let $c\geq 0$ and $\kappa>0$. For the Driver derivative of $V^{c,\kappa}$ in $\phi \in \mathcal{X}$ \ac{and for any $\mathfrak u\in\mathbb R$}, we obtain (following the detailed calculations provided in \cite[Example~1]{CKP23}):
    \begin{align}       
     \nonumber
        &\hspace{-4mm} D^+ V^{c,\kappa} (\phi,\mathfrak{u})\\
        &\hspace{-4mm}= -2\kappa\phi^2(0)
        + 2\kappa\phi(0)\phi(-1)
        - 2\kappa\psi(\phi^2(0) - \mathfrak{u}^2)\phi^2(0)
        \nonumber\\
        \nonumber
        & + \phi^2(0) - e^{-c}\phi^2(-1) -
        c \int_{-1}^0 e^{cs}\phi^2(s) ds
        \\
        \nonumber
        &\hspace{-4mm}=-\left( 2\kappa\psi(\phi^2(0) -\mathfrak{u}^2) - e^c\kappa^2-1+2\kappa\right)\phi^2(0)
        \\
        \label{eq:exampleVdot}
        & - \left(e^{c/2}\kappa\phi(0)-e^{-c/2}\phi(-1)\right)^2
 -
        c \int_{-1}^0 e^{cs}\phi^2(s) ds.
    \end{align}
    In particular, for $c=0$ and $\kappa=1$, we get
    \begin{align}
    \label{eq-13new}
        &\hspace{-3mm}D^+V^{0,1}(\phi,\mathfrak{u})\nonumber\\
        &=-2\psi(\phi^2(0) - \mathfrak{u}^2)\phi^2(0)
        - \left(\phi(0)-\phi(-1)\right)^2
        \leq 0.
    \end{align}
    This shows that $V^{0,1}$ is a UGS Lyapunov function with zero Lyapunov gain and thus system \eqref{eq:examplesys} is $V^{0,1}$-UGS with zero gain \ac{(see Proposition \ref{prop:UGS_LK_theorem})}.
    In addition, observe that
\begin{align*}
    |\phi(0)|\geq 2|\mathfrak u|\quad\Rightarrow\quad \psi(\phi^2(0)-\mathfrak u^2)\geq \tilde \psi(|\phi(0)|),
\end{align*}
where $\tilde\psi(s):=\min_{|r|\leq s/2}\psi(s^2- r^2) = \min\{\psi([\frac{3}{4}s^2,s^2])\}$, for all $s\geq 0$. Clearly, $\tilde\psi \in \mathcal P$, and we get from \eqref{eq-13new} that
    \begin{equation}
      \hspace{-3mm}  |\phi(0)| \geq 2|\mathfrak u| \ \ \Rightarrow \ \  D^+V^{0,1}(\phi, \mathfrak{u}) \leq - 2\tilde\psi(|\phi(0)|)\phi^2(0).
    \end{equation}
    Thus, by Remark~\ref{rem:DinivsDriver}, $V^{0,1}$ is an ISS LKF with pointwise implication form (Definition~\ref{def:ISS_LKF-implication}). With Theorem~\ref{thm:Krasovskii-Theorem}, we conclude that the system \eqref{eq:examplesys} is $V^{0,1}$-ISS and ISS, which proves item $\ref{itm:positive_result}$. 
    
    \underline{Proof of $(ii)$.} \ac{This follows readily from item $\ref{itm:positive_result}$ and Proposition \ref{prop:V-ISS_and_ISS}.}

    \underline{Proof of $(iii)$.} Fix a pair $(c,\kappa) \neq (0,1)$. As $\psi$ is $C^\infty$ and has constant value zero on $(-\infty,0]$, all its derivatives vanish \amc{at} $0$. Thus there exist constants $M,C>0$, such that $0\leq \psi(s) \leq Cs^2$ whenever $|s|\leq M$. Introducing the constant 
    \[
    K := e^c\kappa^2+1-2\kappa = (e^c -1)\kappa^2 + (\kappa-1)^2 >0,
    \]
    where we use $(c,\kappa) \neq (0,1)$, we get from \eqref{eq:exampleVdot} that
    \begin{multline}      
    \label{eq:exampleVdot2}
      D^+V^{c,\kappa}(\phi,\mathfrak u) = - 2\kappa\psi(\phi^2(0) - \mathfrak u^2)\phi^2(0)
        + K \phi^2(0)  \\
         - \left(e^{c/2}\kappa\phi(0) - e^{-c/2}\phi(-1)\right)^2 -
        c \int_{-1}^0 e^{cs}\phi^2(s) ds.
     \end{multline}

Consider $\phi \in \mathcal{X}$ such that the following conditions are satisfied: (a) $0< \phi^2(0) \leq M$,  (b) $6\kappa C \phi^4(0) \leq K$, (c) $\phi(-1) = e^c \kappa \phi(0)$, (d) $3c \int_{-1}^0 e^{cs}\phi^2(s) ds\leq K \phi^2(0)$. Then, \ac{using successively (c), (a), (d) and (b),} \eqref{eq:exampleVdot2} for $\mathfrak u=0$ yields
   \begin{align*}
       D^+ V^{c,\kappa}(\phi,0) &=
       - 2\kappa \psi(\phi^2(0))\phi^2(0)  
        + K \phi^2(0)  \\
        & \hspace*{0.4cm}  -
        c \int_{-1}^0 e^{cs}\phi^2(s) ds.\\
        &
        \geq - 2\kappa C \phi^6(0) + K \phi^2(0) 
        -        c \int_{-1}^0 e^{cs}\phi^2(s) ds  \\
        &
        \geq \left(K - 2\kappa C \phi^4(0) - \frac{K}{3}\right) \phi^2(0) 
        >0.
    \end{align*}
    With the choice of $\phi$ satisfying (a)--(d), it is clear that for all $\delta\in (0,1]$ also $\delta \phi$ satisfies (a)--(d). We thus obtain that $D^+ V^{c,\kappa}(\delta\phi,0) >0$, which establishes $(iii)$. 

    \underline{Proof of $(iv)$.} By $(iii)$, the only possibility to obtain \eqref{eq:ISS-LK-implication-V} or \eqref{eq:ISS-LK-sum-LKFwise} is to choose $c=0, \kappa=1$. However, in this case the equality \eqref{eq-13new} holds. Choosing $\phi \in \mathcal{X}\setminus\{0\}$ with $\phi(-1) = \phi(0) =0$, we obtain $D^+ V^{0,1}(\phi,\mathfrak{u}) =0$ for all $\mathfrak{u} \in\R$. However, both the conditions \eqref{eq:ISS-LK-implication-V} \amc{and} \eqref{eq:ISS-LK-sum-LKFwise} require that $D^+ V^{0,1}(\phi,0) < 0$. It is therefore impossible to satisfy \eqref{eq:ISS-LK-implication-V} or \eqref{eq:ISS-LK-sum-LKFwise} for all $\phi\in\mathcal X$ and $\mathfrak{u}\in\R$.

\underline{Proof of $(v)$.} Again, by $(iii)$, we only need to consider the case  $c=0$ and $\kappa=1$. Fix an arbitrary $\chi\in\Kinf$. By choosing $\phi\in\mathcal{X}$ with $0< \phi(0) = \phi(-1)$ and  $\mathfrak u=\phi(0)$, we have from \eqref{eq-13new} that $D^+ V^{0,1}(\phi,\mathfrak{u}) =0$.  Now using the integral term in \eqref{eq:Vck-def}, we can ensure that $V(\phi) \geq \chi(|\mathfrak{u}|)$, which establishes $\ref{item:wang}$.

\underline{Proof of $(vi)$.} By $(iii)$, we only need to consider the case  $c=0, \kappa=1$. Choosing $\mathfrak u=0$  and $\phi\in\mathcal{X}$ with $0< \phi(0) = \phi(-1)$, we have by    \eqref{eq-13new} that $D^+V^{0,1}(\phi,0)=-2\psi(\phi^2(0))\phi^2(0)$. Since $\psi(s)s\to0$ as $s\to\infty$, we see  that $D^+{V}^{0,1}(r\phi,0) \to 0$ as $r\to \infty$. Thus, for every $\alpha\in \Kinf$ and \amc{$Q:\R\to \R_+$} positive definite and radially unbounded, the condition \eqref{eq:epiphane} does not hold, as this would require that $D^+{V}^{0,1}(r\phi,0) \to -\infty$ as $r\to \infty$.
\end{proof}

\section{\texorpdfstring{$V$-stability theory}{V-stability theory}}\label{sec:Vstability}

On our way to proving Theorem \ref{thm:Krasovskii-Theorem}, we develop \ac{some theory for} $V$-stability, which captures the 
behavior of solutions not in terms of the classical $\|\cdot\|$-norm, but rather through a particular LKF candidate $V$. This section  \ac{introduces the needed $V$-stability notions} 
and provides a superposition theorem for $V$-ISS 
in terms of weaker properties.

\subsection{\texorpdfstring{$V$-ISS superposition theorem}{V-ISS superposition theorem}}

\ac{We start by introducing a uniform limit property, under which solutions are requested to approach a neighborhood of the origin (with size ``proportional'' to the input magnitude) within a time that depends only on the norm of the initial state and the applied input.}

\vspace{2mm}
\begin{definition}[ULIM, $V$-ULIM]
\label{def:VULIM}
Given an LKF candidate $V:\ \mathcal X^n\to\R_+$, the system \eqref{eq:time-delay} is said to have the \emph{$V$-uniform limit property ($V$-ULIM)}, if there exists
    $\gamma\in\Kinf\cup\{0\}$ so that, for every $\eps,r>0$, there
    exists $\tau = \tau(\eps,r) \geq 0$ such that, 
for all $x_0 \in {\mathcal X^n}$ with $V(x_0) \leq r$ and all $u\in\Uc$ with $\|u\| \leq r$, there is $t \in[0,\tau]$ so that 
\begin{align}
V(x_t(x_0,u)) \leq \eps + \gamma(\|u\|).
\label{eq:ULIM_ISS_section}
\end{align}
The system \eqref{eq:time-delay} is said to have the \emph{uniform limit property (ULIM)} if the above holds with $V(\cdot)=\|\cdot\|$. \ac{The functions $\gamma$ is referred to as a \emph{$V$-ULIM or ULIM gain}.}
\end{definition}
\vspace{2mm}

 The ULIM notion shares some similarities with the more classical LIM property \ac{originally introduced in a finite-dimensional setting \cite{SoW96} and later extended to infinite-dimensional systems \cite{MiP20}. The key difference lies in the fact that, here, }the maximal time needed for \eqref{eq:ULIM_ISS_section} to hold is required to be uniform on bounded balls of both initial states and inputs. \ac{In light of the counterexample in \cite{mancilla2024forward}, this uniformity does not come for free in the context of time-delay systems, even in the absence of inputs.}

\amc{
\begin{theorem}[\ac{$V$-ISS $\Leftrightarrow$ $V$-ULIM \& $V$-UGS}]
\label{theo:ULIM_plus_GS_implies_ISS}
Let $V:\mathcal X^n\to\R_+$ be an LKF candidate. \ac{Under Assumptions \ref{ass:Regularity-of-f} and \ref{ass:0-is-an-equilibrium}, the system}
\eqref{eq:time-delay} is $V$-ISS if and only if it is both $V$-ULIM and $V$-UGS.
\end{theorem}

\ac{If $V=\|\cdot\|$, Theorem~\ref{theo:ULIM_plus_GS_implies_ISS} is a special case of \cite[Theorem 5]{MiW18b}, which in turn is an extension to the infinite-dimensional setting of the superposition theorem presented in \cite{SoW96} in the context of ODE systems.
It is possible to prove stronger characterizations of $V$-ISS similar to those in \cite[Theorem 5]{MiW18b}, but this is not needed for our purposes here. 
In contrast to the superposition theorems for ODEs \cite{SoW96}, 
due to the infinite-dimensional nature of the considered dynamics, in Theorem~\ref{theo:ULIM_plus_GS_implies_ISS} the LIM property is required 
 to be uniform on bounded sets of the initial states and the inputs.
}


\vspace{2mm}
\begin{proof} (of Theorem \ref{theo:ULIM_plus_GS_implies_ISS})
\ac{The fact that $V$-ISS implies both $V$-ULIM and $V$-UGS is straightforward, so we focus on the converse. Assuming that \eqref{eq:time-delay} is $V$-UGS, there exist $ \sigma\in\mathcal K_\infty$ and $\gamma_1\in \Kinf\cup\{0\}$ such that, for all $ x_0 \in \mathcal X^n$ and  all $u
          \in \Uc$,
\begin{equation}\label{eq-5}
V(x_t(x_0,u)) \leq \sigma(V(x_0)) + \gamma_1(\|u\|) \quad \forall t \geq 0.
\end{equation}   
Assuming that \eqref{eq:time-delay}  also has the $V$-ULIM property, let $\gamma_2\in\Kinf\cup\{0\}$ denote the corresponding gain in \eqref{eq:ULIM_ISS_section}.} Define
\[
\gamma(s):=\max\{\gamma_1(s),\gamma_2(s)\},
\quad  s\geq0.
\]
Then $\gamma\in\Kinf\ac{\cup\{0\}}$ can be used in the estimates both
for $V$-ULIM and $V$-UGS. Consider also the $\Kinf$ function $\tilde\gamma$ defined as
\[
\tilde\gamma(s):=\sigma(2\gamma(s))+\gamma(s)+\sigma(s),
\quad  s\geq0.
\]
Fix any $\eps,r>0$ and let
$\tilde\eps:=\tfrac12\sigma^{-1}(\eps)>0$.
By the $V$-ULIM property, there exists
$\tau=\tau(\tilde\eps,r)\geq0$ such that, for all
$x_0\in\mathcal X^n$ with $V(x_0)\leq r$ and all
$u\in\Uc$ with $\|u\|\leq r$, there is a
$t_1\in[0,\tau]$ satisfying
\[
V(x_{t_1}(x_0,u))\leq\tilde\eps+\gamma(\|u\|).
\]
In view of the cocycle property and $V$-UGS, for any such
$x_0,u$ and any $t\geq\tau$ we obtain
\begin{align*}
V(x_t(x_0,u))
&=V\big(x_{t-t_1}(x_{t_1}(x_0,u),u(t_1+\cdot))\big)\\
&\stackrel{\eqref{eq-5}}{\leq}\sigma\big(V(x_{t_1}(x_0,u))\big)+\gamma(\|u\|)\\
&\leq\sigma\big(\tilde\eps+\gamma(\|u\|)\big)
      +\gamma(\|u\|)\\
&\leq\sigma(2\tilde\eps)
      +\sigma(2\gamma(\|u\|))+\gamma(\|u\|)\\
&\leq\eps+\tilde\gamma(\|u\|),
\end{align*}
where we used $\|u(t_1+\cdot)\|\leq\|u\|$ and the
inequality $\sigma(a+b)\leq\sigma(2a)+\sigma(2b)$ valid for all $a,b\geq0$.

If instead $\|u\|>r$, then $V(x_0)\leq r<\|u\|$ and $V$-UGS give, for all $t\geq0$,
\begin{align*}
V(x_t(x_0,u))
\leq\sigma(V(x_0))+\gamma(\|u\|)
&\leq\sigma(\|u\|)+\gamma(\|u\|)\\
&\leq\tilde\gamma(\|u\|).
\end{align*}
Thus, for every $\eps,r>0$, there is
$\tau\geq0$ such that, for every $u\in\Uc$,
\begin{align}\label{eq-4}
V(x_0)\leq r,\ t\geq\tau
\quad\Rightarrow\quad
V(x_t(x_0,u))\leq\eps+\tilde\gamma(\|u\|).
\end{align}
For each $r,t\geq0$, define
\begin{align}\label{eq-3}
\hat\beta(r,t):=\sup\big\{&
\max\{V(x_s(x_0,u))-\tilde\gamma(\|u\|),0\}:\\
&x_0\in\mathcal X^n,\ V(x_0)\leq r,\ u\in\Uc,\ s\geq t\big\}.\nonumber
\end{align}
Since $\tilde\gamma\geq\gamma\geq \gamma_1$, the $V$-UGS estimate \eqref{eq-5} ensures that
\[
\hat\beta(r,t)\leq\sigma(r),
\quad \forall r,t\geq0.
\]
Hence $\hat\beta$ is \ac{locally bounded and satisfies} $\hat\beta(0,t)=0$ and
$\lim_{r\to 0}\hat\beta(r,t)=0$ uniformly in $t\geq0$.
By construction, $\hat\beta$ is nondecreasing in $r$
and nonincreasing in $t$. Moreover, 
\ac{the estimate \eqref{eq-4}} shows that, for every fixed $r,\varepsilon>0$, we have $\hat\beta(r,t)\leq\eps$ for
all sufficiently large $t$. Thus
$\lim_{t\to+\infty}\hat\beta(r,t)=0$. By \cite[Proposition A.17]{Mir23}, these combined properties guarantee that there exists
$\beta\in\KL$ such that $\hat\beta(r,t)\leq\beta(r,t)$ for all $r,t\geq 0$.
Taking $r=V(x_0)$ and $s=t$, we see from \eqref{eq-3} that
\[
V(x_t(x_0,u))
\leq\beta(V(x_0),t)+\tilde\gamma(\|u\|),
\quad \forall t\geq0,
\]
for all $x_0\in\mathcal X^n$ and all $u\in\Uc$, and $V$-ISS follows.
\end{proof}
}

\subsection{Bounds on solutions}

We finally present technical results that provide bounds on the solutions (and their derivatives) of a $V$-UGS system.

\ac{Our first observation is that,} while the definition of $V$-UGS provides an upper bound on $V(x_t)$, a bound on the whole history norm can also be obtained after one full delay period.

\vspace{2mm}
\begin{proposition}[\ac{State bound from $V$-UGS}]\label{prop:V-UGS_and_V-ISS} 
Given an LKF candidate $V:\mathcal X^n\to\R_+$, if \eqref{eq:time-delay} is $V$-UGS then there are $\sigma\in\Kinf$ \ac{and $\gamma\in\Kinf\cup\{0\}$} such that, for all $x_0 \in {\mathcal X^n}$ and all $u\in\Uc$,
\begin {equation}
\label{eq:V-to-phi-UGS}
\|x_t(x_0,u)\| \leq \sigma(V(x_0)) + \gamma( \|u\|),\quad \forall t\geq \theta.
\end{equation}
\ac{Moreover, $\gamma$ can be chosen to be the zero function if the $V$-UGS gain is itself identically zero.}
\end{proposition}
\vspace{2mm}

\begin{proof}
Since $V$ is an LKF candidate and \eqref{eq:time-delay} is $V$-UGS, we may apply \eqref{eq:V-GSAbschaetzung} and \eqref{eq:Sandwich} to conclude that there are $\psi_1 ,\tilde\sigma\in\Kinf$ \ac{and $\tilde\gamma\in\Kinf\cup\{0\}$} so that, for all $x_0 \in {\mathcal X^n}$ and all $u\in\Uc$, 
\begin{eqnarray*}
\psi_1(|x(t,x_0,u)|) \leq \tilde\sigma(V(x_0)) + \tilde\gamma( \|u\|)\quad\forall t\geq 0.
\end{eqnarray*}
As $\psi_1^{-1}(a+b) \leq \psi_1^{-1}(2a) + \psi_1^{-1}(2b)$ for all $a,b\geq 0$, we get
\begin{equation*}
|x(t,x_0,u)| \leq \psi^{-1}_1\big(2\tilde\sigma(V(x_0))\big) + \psi^{-1}_1\big(2\tilde\gamma( \|u\|)\big) \quad \forall t\geq 0.
\end{equation*}
Consequently, for all $t\geq \theta$,
\begin{align*}
	\|x_t(x_0,u)\| =  \max_{\tau \in [-\theta,0]}|x(t+\tau,x_0,u)| \leq	 \sigma(V(x_0)) +\gamma(\left\|u\right\|),
\end{align*}
with $\sigma:=\psi_1^{-1}\circ 2\tilde\sigma$ and $\gamma:=\psi_1^{-1}\circ 2\tilde\gamma$. \ac{The final conclusion follows by noticing that $\gamma\equiv 0$ if $\tilde\gamma\equiv 0$.}
\end{proof}

\vspace{2mm}
$V$-UGS also provides a bound on the solutions' derivative after a full delay period. To establish this fact, we first make the following observation.

\vspace{2mm}
\begin{proposition}[Bound on the vector field]
\label{prop:Bounds_on_f}
If Assumptions~\ref{ass:Regularity-of-f} and \ref{ass:0-is-an-equilibrium} hold, then \amc{$f$ is \emph{$\Kinf$-bounded}, meaning that there exist $\xi_1,\xi_2 \in\Kinf$ such that}
\begin{eqnarray}
|f(\phi,v)| \leq \xi_1(\|\phi\|) + \xi_2(|v|)\quad \forall\phi\in {\mathcal X^n},\,v\in\R^m.
\label{eq:estimate_f}
\end{eqnarray}
\end{proposition}
\vspace{2mm}

\amc{
\begin{proof}
Due to Lipschitz continuity of $f$ on bounded balls w.r.t. the first argument \lb{(Assumption~\ref{ass:Regularity-of-f})}, there is a strictly increasing continuous function $L$ (characterizing a Lipschitz constant of $f$) so that 
for all $r>0$ and all $\phi\in {\mathcal X^n}$ with $\|\phi\|\leq r$ and $v \in\R^m$ with $|v|\leq r$, we have
\begin{multline*}
|f(\phi,v)| \leq |f(0,v)| + |f(\phi,v) - f(0,v)|\\
\leq \sup_{|v|\leq r}|f(0,v)| + L(r)\|\phi\|
\leq \sup_{|v|\leq r}|f(0,v)| + L(r)r.
\end{multline*}

Define $z(r):= r + \sup_{\|\phi\|\leq r,\ |v|\leq r}|f(\phi,v)|$, $r\geq 0$. Then $z$ is nondecreasing and $\lim_{r\to +0}z(r)=0=z(0)$.
By \cite[Proposition A.16]{Mir23}, there is $\alpha \in\Kinf$ such that $z(r)\leq \alpha(r)$, $r\ge 0$. 
Finally, for all $\phi\in {\mathcal X^n}$ and $v \in\R^m$ we have\\
$\phantom{aaaa}|f(\phi,v)|\leq z(\max\{\|\phi\|,|v|\})\leq \alpha(\|\phi\|) + \alpha(|v|).$
\end{proof}
}

\vspace{2mm}
Based on Proposition~\ref{prop:Bounds_on_f}, we have the following.
\vspace{2mm}

\begin{lemma}[Bounds on solution derivatives]
\label{lem:Global_Bounds_on_derivative-V-UGS}
Given an LKF candidate $V:\mathcal X^n\to\R_+$, assume that \eqref{eq:time-delay} is $V$-UGS and satisfies Assumptions~\ref{ass:Regularity-of-f} and \ref{ass:0-is-an-equilibrium}. 
Then there exist $\mu_1, \mu_2\in\Kinf$ such that, for all $x_0\in {\mathcal X^n}$ and all $u\in\Uc$,
\begin{align}
\label{eq:Global_Bounds_on_derivative-V}
|\dot x(t,x_0,u)|\leq \mu_1(V(x_0)) + \mu_2(\|u\|)\quad \forall t\geq \theta \,\, a.e.
\end{align}
\end{lemma}
\vspace{2mm}

\begin{proof}
By Proposition \ref{prop:V-UGS_and_V-ISS}, there exist $\sigma_1,\sigma_2\in\Kinf$ such that, for all $x_0\in {\mathcal X^n}$ and all $u\in\Uc$,
\begin{equation*}
\left\| x_t(x_0,u) \right\| \leq \sigma_1(V(x_0)) + \sigma_2(\|u\|)\quad \forall t\geq \theta.
\end{equation*}
Let $\xi_1,\xi_2 \in\Kinf$ satisfy the assertion of Proposition \ref{prop:Bounds_on_f}.
Fix $x_0\in {\mathcal X^n}$ and $u\in\Uc$. \amc{Then it} follows from \eqref{eq:estimate_f} that, for almost all $t\geq \theta$,
\begin{align*}
|\dot x(t,x_0,u)| &= \big|f(x_t(x_0,u),u(t))\big|\\
&\leq \xi_1(\|x_t(x_0,u)\|) + \xi_2(|u(t)|)\\
&\leq \xi_1\big(\sigma_1(V(x_0)) + \sigma_2(\|u\|)\big) + \xi_2(\|u\|)\\
&\leq \xi_1\big(2\sigma_1(V(x_0))\big) + \xi_1\big(2\sigma_2(\|u\|)\big) + \xi_2(\|u\|)
\end{align*}
and the claim follows with $\mu_1:=\xi_1\circ 2\sigma_1$ and $\mu_2:=\xi_1\circ 2\sigma_2 + \xi_2$.
\end{proof}

\section{Proof of Theorem \ref{thm:Krasovskii-Theorem} }\label{sec:proof}

Now we can establish our main result (Theorem \ref{thm:Krasovskii-Theorem}).

We restate it here for the sake of readability.

\vspace{2mm}
\textbf{Theorem~\ref{thm:Krasovskii-Theorem}} \emph{Let Assumptions~\ref{ass:Regularity-of-f} and \ref{ass:0-is-an-equilibrium} hold. 
If there exists an LKF candidate $V:\mathcal X^n\to\R_+$ which is simultaneously an ISS LKF with pointwise dissipation (in either implication or \lb{dissipative form}) and a UGS LKF for \eqref{eq:time-delay}, then \eqref{eq:time-delay} is $V$-ISS and, in particular, ISS.}
\vspace{2mm}

\begin{proof}
In view of Proposition \ref{prop-sum-imp}, it is enough to assume that $V$ is both an ISS LKF in pointwise implication form and a UGS LKF. 
Furthermore, by Proposition~\ref{prop:Kinf-decay-rate-always-suffices-pointwise-dissipation},  we can assume that the decay rate of $V$ is a $\Kinf$-function (note that the rescaled function constructed in Proposition~\ref{prop:Kinf-decay-rate-always-suffices-pointwise-dissipation} will also be a UGS-LKF for \eqref{eq:time-delay}, and it would still certify $V$-ISS with the original $V$ in view of Proposition~\ref{prop:V-ISS-Kinf-invariance}).

The proof uses the $V$-ISS superposition theorem (Theorem~\ref{theo:ULIM_plus_GS_implies_ISS}) and shows that \eqref{eq:time-delay} is both $V$-UGS and $V$-ULIM. The former follows directly from Proposition~\ref{prop:UGS_LK_theorem}. For the latter, we first fix $\chi\in\Kinf$ as a Lyapunov gain as in Definitions~\ref{def:noncoercive_UGS_LF} and \ref{def:ISS_LKF-implication} (if the Lyapunov gains are different, we can define $\chi$ as the maximum of the two). Also let $\psi_2\in\Kinf$ be an upper bound on $V$ as in the sandwich condition \eqref{eq:Sandwich}. \amc{Without loss of generality, we can assume that $\psi_2(r) \geq r$ for all $r\ge 0$.}

\ac{We will show that the system \eqref{eq:time-delay} is $V$-ULIM with gain $\gamma:=\psi_2\circ 2\chi$. To that aim, we proceed by contradiction and assume, to the contrary, that} the system does not satisfy the $V$-ULIM condition \eqref{eq:ULIM_ISS_section} with this gain $\gamma$ and some function $\tau=\tau(\varepsilon,r)\geq 3\theta$ to be defined later.
Hence, there are some $r,\varepsilon>0$, some $x_0 \in {\mathcal X^n}$ with $V(x_0)\le r$, and some $u\in\Uc$ with $\|u\|\leq r$ such that 
\begin{eqnarray}
V(x_t(x_0,u)) \geq \varepsilon + \gamma(\|u\|) \quad \forall t\in[0,\tau(\varepsilon,r)].
\label{eq:Anti-ULIM}
\end{eqnarray}
By \eqref{eq:Sandwich}, it follows that 
\begin{align*}
\psi_2(\|x_t(x_0,u)\|) \geq \varepsilon + \psi_2\circ 2\chi(\|u\|) \quad \forall t\in[0,\tau(\varepsilon,r)],
\end{align*}
which in turn implies that
\begin{equation}\label{eq:Anti-ULIM-3}
\hspace{-2.75mm}\|x_t(x_0,u)\| \geq \max\big\{\psi_2^{-1}(\varepsilon), 2\chi(\|u\|)\big\}, \hspace{1.25mm} t\in[0,\tau(\varepsilon,r)].
\end{equation}
\amc{Using $\tau(\varepsilon,r)\geq 3\theta$, let $K\in \mathbb N$  be the integer with
\begin{equation}\label{eq-K}
\tau(\varepsilon,r)-3\theta < 2\theta K \leq \tau(\varepsilon,r)-\theta.
\end{equation}

\begin{figure}[htbp]
    \centering
\resizebox{\dimexpr\linewidth-4pt\relax}{!}{%
\begin{tikzpicture}[
    x=0.6cm, y=1cm,
    font=\normalsize,
    axis/.style={draw=blue!55!black,line width=0.85pt},
    interval/.style={draw=red!75!black,line width=0.8pt},
    interval label/.style={text=red!75!black,inner sep=1.5pt},
    time label/.style={text=blue!55!black,inner sep=1.5pt}
]
\draw[axis] (0,0) -- (16.4,0);
\foreach \x in {0,1.8,3.6,5.4,7.2,12.6,14.4,16.4}
    \draw[axis] (\x,-0.11) -- (\x,0.11);

\node[time label,above=5pt] at (0,0) {$0$};
\node[time label,below=7pt] at (1.8,0) {$\theta$};
\node[time label,below=7pt] at (3.6,0) {$2\theta$};
\node[time label,below=7pt] at (5.4,0) {$3\theta$};
\node[time label,below=7pt] at (7.2,0) {$4\theta$};
\node[time label,below=7pt,xshift=-4pt] at (12.6,0) {$(2K-1)\theta$};
\node[time label,below=7pt,xshift=4pt] at (14.4,0) {$2K\theta$};
\node[time label,above=5pt] at (16.4,0) {$\tau(\varepsilon,r)$};

\foreach \a/\b in {2.35/3.75,6.55/7.95,12.75/14.15} {
    \fill[red!9] (\a,-0.065) rectangle (\b,0.065);
    \fill[pattern=north east lines,pattern color=red!75!black]
        (\a,-0.065) rectangle (\b,0.065);
    \draw[interval] (\a,0) -- (\b,0);
    \draw[interval] (\a,-0.12) -- (\a,0.12);
    \draw[interval] (\b,-0.12) -- (\b,0.12);
}
\foreach \x in {3.6,7.2}
    \draw[axis] (\x,-0.11) -- (\x,0.11);

\node[interval label,above=6pt] at (2.35,0) {$t_1$};
\node[interval label,above=6pt,xshift=4pt] at (3.75,0) {$t_1+\delta$};
\node[interval label,above=6pt] at (6.55,0) {$t_2$};
\node[interval label,above=6pt,xshift=4pt] at (7.95,0) {$t_2+\delta$};
\node[interval label,above=6pt] at (12.75,0) {$t_K$};
\node[interval label,above=6pt,xshift=4pt] at (14.15,0) {$t_K+\delta$};

\node[interval label] (I1) at (3.45,1.25) {$I_1$};
\node[interval label] (I2) at (7.65,1.25) {$I_2$};
\node[interval label] (IK) at (13.95,1.25) {$I_K$};
\draw[interval,-{Stealth[length=1.8mm,width=1.3mm]}]
    (I1.south west) to[out=225,in=90] (2.93,0.12);
\draw[interval,-{Stealth[length=1.8mm,width=1.3mm]}]
    (I2.south west) to[out=225,in=90] (7.13,0.12);
\draw[interval,-{Stealth[length=1.8mm,width=1.3mm]}]
    (IK.south west) to[out=225,in=90] (13.43,0.12);

\node[time label] at (10.35,-0.38) {$\cdots$};
\end{tikzpicture}%
}%
    \caption{Time instants $t_k$ and intervals
        $I_k=[t_k,t_k+\delta)$ used in the proof of Theorem~\ref{thm:Krasovskii-Theorem}.}
    \label{fig:proof-timeline}
\end{figure}

By \eqref{eq:Anti-ULIM-3}, there exists an increasing finite sequence of time instants $t_k \in [(2k-1)\theta,2k\theta]$, $k\in\{1,\ldots, K\}$, such that 
}
\begin{eqnarray}
|x(t_k,x_0,u)| \geq \max\big\{\psi_2^{-1}(\varepsilon), 2\chi(\|u\|)\big\}.
\end{eqnarray}
By Lemma~\ref{lem:Global_Bounds_on_derivative-V-UGS}, there exist $\mu_1,\mu_2\in\Kinf$ such that 
\begin{eqnarray}
\label{eq:Global_Bounds_on_derivative-aux}
\hspace{-7mm}|\dot x(t,x_0,u)| \!\le\! \mu_1(V(x_0)) + \mu_2(\|u\|) \leq \mu(r),\, \ t\ge \theta\,\,\text{a.e.},
\end{eqnarray}
where $\mu:=\mu_1 +\mu_2$. \amc{For each $k\in\{1,\ldots, K\}$, consider the interval
\[
I_k:=\left[ t_k, t_k + \delta\right),\quad \delta := \min\Big\{\theta, \frac{\psi_2^{-1}(\varepsilon)}{2\mu(r)}\Big\}.
\]
With this choice of $\delta$, these intervals do not overlap, see Figure~\ref{fig:proof-timeline} for an illustration.
}
In view of \eqref{eq:Global_Bounds_on_derivative-aux}, for all $k\in\{1,\ldots,K\}$ and all $t \in I_k$, we have 
\begin{align*}
|x(t,x_0,u)| \geq \max\left\{\frac{\psi_2^{-1}(\varepsilon)}{2}, 2\chi(\|u\|)- \frac{\psi_2^{-1}(\varepsilon)}{2}\right\}.
\end{align*}
Note that if $c>\max\{a,2b-a\}$ for some $a,b,c\geq 0$, then $c>\max\{a,b\}$ (consider $a>b$ and $a\le b$). It follows that
\begin{eqnarray}
|x(t,x_0,u)| \geq \max\left\{\frac{1}{2}\psi_2^{-1}(\varepsilon), \chi(\|u\|)\right\}  
\quad \forall t \in I_k.
\label{eq:Anti-ULIM-5}
\end{eqnarray}
Since \ac{$\gamma=\psi_2\circ 2\chi$ and} $\psi_2(s)\geq s$ for all $s\ge 0$, \eqref{eq:Anti-ULIM} gives
\begin{align*}
V(x_t(x_0,u)) \geq \chi(\|u\|) \quad \forall t\in[0,\tau(\varepsilon,r)].
\end{align*}
Hence, \ac{the UGS dissipation inequality}  \eqref{DissipationIneq_UGS_LK_functionals} yields that 
\begin{eqnarray}
\dot{V}_{u(t+\cdot)}(x_t(x_0,u))\leq 0\quad  \forall t\in[0,\tau(\varepsilon,r)].
\label{eq:Derivative-nonpositiveness}
\end{eqnarray}
Using first \cite[Lemma 3.4]{MiW19a} and then \eqref{eq:Derivative-nonpositiveness}, it follows that
\begin{align*}
V(x_{\tau(\varepsilon,r)}(x_0,u))-V(x_0) &\le \int_0^{\tau(\varepsilon,r)} \dot{V}_{u(t+\cdot)}(x_t(x_0,u))dt\\
& \le \sum_{k=1}^{K}\int_{I_k}
 \dot{V}_{u(t+\cdot)}(x_t(x_0,u))dt.
\end{align*}
Using \amc{successively the pointwise ISS dissipation \eqref{eq:ISS-LK-implication-relaxed}, \eqref{eq:Anti-ULIM-5} \ac{and \eqref{eq-K}}, we obtain
\begin{align*}
V(x_{\tau(\varepsilon,r)}(x_0,u))&-V(x_0)
\stackrel{\eqref{eq:ISS-LK-implication-relaxed}}{\le} -\sum_{k=1}^K \int_{I_k}
 \alpha (|x(t,x_0,u)|)dt\\
&\stackrel{\eqref{eq:Anti-ULIM-5}}{\le} -\sum_{k=1}^K \int_{I_k}
 \alpha\left(\frac{1}{2}\psi_2^{-1}(\varepsilon)\right)dt\\
 &= -K\delta\alpha\left(\frac{1}{2}\psi_2^{-1}(\varepsilon)\right)\\
 &\stackrel{\eqref{eq-K}}{\le} -\frac{\tau(\varepsilon,r)-3\theta}{2\theta}\delta\alpha\left(\frac{1}{2}\psi_2^{-1}(\varepsilon)\right). 
\end{align*}
This ensures in particular that 
\begin{eqnarray}
r \geq V(x_0) \geq  \frac{\tau(\varepsilon,r)-3\theta}{2\theta}\delta\alpha\left(\frac{1}{2}\psi_2^{-1}(\varepsilon)\right).
\label{eq:Close-to-contradiction}
\end{eqnarray}
For the choice
\begin{eqnarray}
\tau(\varepsilon,r):= 3\theta +\frac{2\theta(r+1)}{\delta\alpha\left(\frac{1}{2}\psi_2^{-1}(\varepsilon)\right)},
\label{eq:Ulim-time-definition-u=0}
\end{eqnarray}
}
\eqref{eq:Close-to-contradiction} yields the contradiction $r> r+1$. Thus,  \eqref{eq:time-delay}  satisfies the $V$-ULIM estimate \eqref{eq:ULIM_ISS_section} with the function $\tau$ given in \eqref{eq:Ulim-time-definition-u=0} and the gain $\gamma=\psi_2\circ2\chi$. Thus the system is both $V$-UGS and $V$-ULIM and $V$-ISS follows from Theorem~\ref{theo:ULIM_plus_GS_implies_ISS}, whereas ISS follows from  Proposition~\ref{prop:V-ISS_and_ISS}. 
\end{proof}

\amc{

\section{Extension of Krasovskii's criterion for uniform global asymptotic stability}
\label{sec:zero-gain-input-uniformity}

\ac{In the case when ISS holds with a zero gain, the global asymptotic stability of the origin is robust, in the sense that it is preserved despite disturbances acting on the system. This interesting feature, often referred to as UGAS, has played a central role in the development of converse Lyapunov theorems for ISS in the finite-dimensional literature \cite{SoW95,LSW96} and has already been used in the time-delay systems literature \cite{KankWang16,KaJ11b}. In this section, we show how our main result can be used to derive a sufficient condition for this property through an LKF with merely pointwise dissipation.}

\subsection{Stability and asymptotic stability semiglobally in the inputs}
\label{subsec:semiglobal-zero-gain}

\ac{The result we present here provides a sufficient condition for $V$-ISS with zero gain valid when the considered set of input values is bounded. To cover the case when this set is unbounded, we introduce the following notion, which allows an arbitrarily large set of input values, at the cost of a possibly slower convergence rate and larger transient overshoots.}



\vspace{2mm}
\begin{definition}[\ac{$V$-UGAS}]
\label{def:V-UGAS-semiglobal-inputs}
Given an LKF candidate $V:\mathcal X^n\to\R_+$, the system~\eqref{eq:time-delay} is called \emph{$V$-UGAS semiglobally in the inputs} if, for every $R>0$, there exists $\beta_R\in\KL$ such that, for all $x_0\in\mathcal X^n$, all $u\in\Uc$ with $\|u\|\leq R$, the solution satisfies
\begin{equation}
\label{eq:V-UGAS-semiglobal-inputs}
V(x_t(x_0,u))\leq\beta_R(V(x_0),t)\quad \forall t\geq 0.
\end{equation}
If $\beta \in\KL$ above can be chosen uniformly in $R\ge 0$, we call \eqref{eq:time-delay} \emph{$V$-UGAS}.
\end{definition}
\vspace{2mm}

\ac{In a similar way, we can introduce the following semiglobal version of UGS with zero gain.}

\vspace{2mm}
\begin{definition}[\ac{$V$-UGS with zero gain}]
\label{def:V-semiglobal-inputs}
Given an LKF candidate $V:\mathcal X^n\to\R_+$, the system~\eqref{eq:time-delay} is called \emph{$V$-UGS with zero gain semiglobally in the inputs} if for every $R>0$ there exists $\sigma_R\in\Kinf$ such that,  for all $x_0\in\mathcal X^n$, all $u\in\Uc$ with $\|u\|\leq R$, the solution satisfies
\begin{equation}
\label{eq:V-UGS-semiglobal-inputs}
V(x_t(x_0,u))\leq\sigma_R(V(x_0))\quad \forall t\geq0.
\end{equation}
\end{definition}
\vspace{2mm}

Note that both these properties, although semiglobal in the input $u$, are global in the initial state $x_0$. In particular, for input-free systems \ac{(or when the considered set of input values is bounded)}, they coincide with $V$-UGAS and $V$-UGS.

\ac{The following result allows to characterize input-semiglobal $V$-UGAS in terms of $V$-ULIM and $V$-UGS properties with zero gain.}


\begin{proposition}
\label{prop:semiglobal-V-UGAS-characterization}
Given an LKF candidate $V:\mathcal X^n\to \R_+$, the following statements are equivalent:
\begin{enumerate}[label=(\roman*)]
\item System~\eqref{eq:time-delay} is $V$-UGAS semiglobally in the inputs.
\item System~\eqref{eq:time-delay} is $V$-ULIM with zero gain and $V$-UGS with zero gain semiglobally in the inputs.
\end{enumerate}
\end{proposition}
\vspace{2mm}

The proof of Proposition~\ref{prop:semiglobal-V-UGAS-characterization} 
 repeats the arguments of the proof of Theorem \ref{theo:ULIM_plus_GS_implies_ISS} and is omitted.

\subsection{Krasovskii's theorem \ac{for systems with inputs}}
\label{subsec:Extension of Krasovskii's theorem}

\ac{We are now ready to state the main result of this section.}


\vspace{2mm}
\begin{corollary}[Extension of Krasovskii's theorem]
\label{cor:Krasovskii-Theorem-zero-gain}
Let Assumptions~\ref{ass:Regularity-of-f} and~\ref{ass:0-is-an-equilibrium} hold. If \eqref{eq:time-delay} admits an ISS LKF $V$ with pointwise dissipation (in implication or dissipative form) with gain $\chi\equiv0$, then the system \eqref{eq:time-delay} is $V$-UGAS semiglobally in the inputs, and $V$-UGS with zero gain.
\end{corollary}
\vspace{2mm}


\ac{A direct consequence of the above result is that, when considering bounded sets of input values, $V$-UGAS (hence UGAS) can be derived by relying on an LKF that dissipates only pointwise, just like in Krasovskii's original result \cite[Theorems 30.2 and 31.3]{Kra63} (see also \cite[Theorem 2.1, p. 132]{HaV93}) in which no exogenous inputs were considered.}

\ac{For unbounded sets of input values, the above result fails to guarantee $V$-UGAS: only its semiglobal counterpart is established, meaning that the $\mathcal{KL}$ estimate of solutions may depend on the magnitude of the considered inputs (although a common bound on the transients is guaranteed by $V$-UGS). This is not an artifact of the proof we employ here, but rather an intrinsic limitation: an example in \cite[Example 7.4]{BCMW26} shows that the existence of an ISS LKF with pointwise dissipation and zero gain does not guarantee $V$-UGAS (or even UGAS). }

\vspace{2mm}
\begin{proof} (of Corollary~\ref{cor:Krasovskii-Theorem-zero-gain})
Let $V$ be an ISS LKF in pointwise implication form and $\chi\equiv 0$.
By Proposition~\ref{prop:Kinf-decay-rate-always-suffices-pointwise-dissipation}, there is $\xi\in\Kinf$ such that $W:=\xi\circ V$ is
an ISS LKF in pointwise implication form with $\chi\equiv 0$ and $\Kinf$-decay rate.

Hence, $W$ is also a UGS LKF with zero gain $\chi$ and its solutions satisfy $W(x_t(x_0,u))\leq W(x_0)$ on their domain of existence. Exploiting the sandwich bounds on $W$ and the BIC property, this ensures $W$-UGS with zero gain.

Following the proof of Theorem~\ref{thm:Krasovskii-Theorem}, we see that \eqref{eq:time-delay} has the $W$-ULIM property with gain $\gamma:=\psi_2 \circ 2\chi = 0$, where $\psi_2$ is the upper sandwich bound for $W$. 

Now Proposition~\ref{prop:semiglobal-V-UGAS-characterization} shows that 
\eqref{eq:time-delay} is $W$-UGAS semiglobally in the inputs. 
Similarly to Proposition~\ref{prop:V-ISS-Kinf-invariance}, we obtain that 
\eqref{eq:time-delay} is also $V$-UGAS semiglobally in the inputs and $V$-UGS with zero gain.
\end{proof}
}

\section{Conclusion and perspectives}

We have shown that ISS can be inferred from an ISS LKF with pointwise dissipation, provided that the same LKF can be used to establish UGS. To this end, we rely on a superposition principle for a variant of ISS, in which solutions are estimated through the LKF rather than through the classical $\|\cdot\|$-norm of the state. 
As evidenced in the example section, the result applies in situations where previously known sufficient conditions do not suffice.

\ac{The main limitation of our result is that showing the considered LKF is indeed a UGS LKF requires expressing non-positivity of its derivative whenever it dominates the input terms. While this non-positivity is clearly less demanding than the negativity required by ISS LKF with LKF-wise dissipation, further work is still needed to enlarge the class of systems for which a pointwise ISS LKF guarantees ISS.}

%
%
%
        %
        %
	%
%
%
%
%



\bibliographystyle{abbrv}
\bibliography{Mir_LitList_NoMir,MyPublications,addition}

@article{CKP23,
  title={The {ISS} framework for time-delay systems: a survey},
  author={Chaillet, Antoine and Karafyllis, Iasson and Pepe, Pierdomenico and Wang, Yuan},
  journal={Mathematics of Control, Signals, and Systems},
  volume={35},
  number={2},
  pages={237--306},
  year={2023},
  publisher={Springer}
}

@ARTICLE{CGP22,
  author={Chaillet, Antoine and Göksu, Gökhan and Pepe, Pierdomenico},
  journal={IEEE Transactions on Automatic Control}, 
  title = {{L}yapunov--{K}rasovskii characterizations of integral input-to-state stability of delay systems with non-strict dissipation rates}, 
  year={2022},
  volume={67},
  number={7},
  pages={3259-3272},
}

@article{SoW00,
  title = {Lyapunov characterizations of input to output stability},
  author = {Sontag, Eduardo and Wang, Yuan},
  journal = {SIAM Journal on Control and Optimization},
  volume = {39},
  number = {1},
  pages = {226--249},
  year = {2000},
  publisher = {SIAM}
}

@phdthesis{Kan17,
  title = {Output stability analysis for nonlinear systems with time delays},
  author = {Kankanamalage, Hasala Gallolu},
  school = {Florida Atlantic University, Boca Raton, FL},
  year = {2017},
  publisher = {Florida Atlantic University}
}

@article{SoW99,
  title = {Notions of input to output stability},
  author = {Sontag, Eduardo D and Wang, Yuan},
  journal = {Systems \& Control Letters},
  volume = {38},
  number = {4-5},
  pages = {235--248},
  year = {1999},
  publisher = {Elsevier}
}

@book{KaK19,
  title = {Input-to-State Stability for {PDE}s},
  author = {Karafyllis, Iasson and Krstic, Miroslav},
  isbn = {9783319910116},
  publisher = {Springer, Cham},
  year = {2019}
}

@article{KLW17,
  title = {On {L}yapunov-{K}rasovskii characterizations of input-to-output stability},
  author = {Kankanamalage, Hasala Gallolu and Lin, Yuandan and Wang, Yuan},
  journal = {IFAC-PapersOnLine},
  volume = {50},
  number = {1},
  pages = {14362--14367},
  year = {2017},
  publisher = {Elsevier}
}

@inproceedings{kankanamalage2019remarks,
  title={Remarks on {Lyapunov-Krasovskii} functionals with weak decay rates},
  author={Kankanamalage, Hasala Gallolu and Lin, Yuandan and Wang, Yuan},
  booktitle={15th IEEE international conference on control and automation (ICCA)},
  pages={21--26},
  year={2019}
}

@article{CPM17,
  title = {Is a point-wise dissipation rate enough to show {ISS} for time-delay systems?},
  author = {Chaillet, Antoine and Pepe, Pierdomenico and Mason, Paolo and Chitour, Yacine},
  journal = {IFAC-PapersOnLine},
  volume = {50},
  number = {1},
  pages = {14356--14361},
  year = {2017},
  publisher = {Elsevier}
}

@article{chaillet2023growth,
  title={Growth conditions for global exponential stability and exp-{ISS} of time-delay systems under point-wise dissipation},
  author={Chaillet, Antoine and Karafyllis, Iasson and Pepe, Pierdomenico and Wang, Yuan},
  journal={Systems \& Control Letters},
  volume={178},
  pages={105570},
  year={2023},
  publisher={Elsevier}
}

@book{Kra63,
  title={Stability of Motion: Applications of {L}yapunov's Second Method to Differential Systems and Equations with Delay},
  author={Krasovskii, Nikolai Nikolaevich},
  isbn={9780804700986},
  lccn={lc63012040},
  url={https://books.google.de/books?id=8g2oAAAAIAAJ},
  year={1963},
  publisher={Stanford University Press}
}

@article{ASW00,
  author = {Angeli, David and Sontag, Eduardo D. and Wang, Yuan},
  title = {A characterization of integral input-to-state stability},
  journal = {IEEE Transactions on Automatic Control},
  year = {2000},
  volume = {45},
  pages = {1082--1097},
  number = {6},
  institution = {Dept. of Syst. \& Inf., Florence Univ., Italy},
  issn = {0018-9286},
  publisher = {IEEE}
}

@article{de2026candidate,
  title={A Candidate Counterexample to a Conjecture on {ISS} for Time-Delay Systems},
  author={De Iuliis, Vittorio and Pepe, Pierdomenico},
  journal={arXiv preprint arXiv:2609.22491},
  year={2026}
}

@book{HaV93,
  title = {Introduction to Functional-Differential Equations},
  publisher = {Springer},
  year = {1993},
  author = {Hale, Jack K. and Verduyn Lunel, Sjoerd M.},
  address = {New York}
}

@book{KaJ11b,
  title = {Stability and Stabilization of Nonlinear Systems},
  publisher = {Springer},
  year = {2011},
  author = {Karafyllis, Iasson and Jiang, Zhong-Ping},
  address = {London}
}

@article{KPJ08,
  author = {Karafyllis, Iasson and Pepe, Pierdomenico and Jiang, Zhong-Ping},
  title = {Input-to-output stability for systems described by retarded functional differential equations},
  journal = {European Journal of Control},
  year = {2008},
  volume = {14},
  pages = {539--555},
  number = {6},
  doi = {10.3166/ejc.14.539-555},
  issn = {09473580},
  url = {http://dx.doi.org/10.3166/ejc.14.539-555}
}

@article{LSW96,
  author = {Lin, Yuandan and Sontag, Eduardo D. and Wang, Yuan},
  title = {A smooth converse {L}yapunov theorem for robust stability},
  journal = {SIAM Journal on Control and Optimization},
  year = {1996},
  volume = {34},
  pages = {124--160},
  number = {1},
  publisher = {SIAM}
}

@article{PeJ06,
  author = {Pepe, P. and Jiang, Zhong-Ping},
  title = {A {L}yapunov--{K}rasovskii methodology for {ISS} and {iISS} of time-delay systems},
  journal = {Systems \& Control Letters},
  year = {2006},
  volume = {55},
  pages = {1006--1014},
  number = {12},
  doi = {10.1016/j.sysconle.2006.06.013},
  issn = {01676911}
}

@incollection{Son08,
  author = {Sontag, Eduardo D.},
  title = {Input to State Stability: Basic Concepts and Results},
  booktitle = {Nonlinear and Optimal Control Theory},
  publisher = {Springer},
  year = {2008},
  chapter = {3},
  pages = {163--220},
  address = {Heidelberg},
  doi = {10.1007/978-3-540-77653-6\_3},
  isbn = {978-3-540-77644-4}
}

@article{Son89,
  author = {Sontag, E. D.},
  title = {Smooth stabilization implies coprime factorization},
  journal = {IEEE Transactions on Automatic Control},
  year = {1989},
  volume = {34},
  pages = {435--443},
  number = {4},
  issn = {00189286}
}

@article{SoW96,
  author = {Sontag, E. D. and Wang, Yuan},
  title = {New characterizations of input-to-state stability},
  journal = {IEEE Transactions on Automatic Control},
  year = {1996},
  volume = {41},
  pages = {1283--1294},
  number = {9},
  issn = {0018-9286},
  publisher = {IEEE}
}

@article{SoW95,
  author = {Sontag, E. D. and Wang, Y.},
  title = {On characterizations of the input-to-state stability property},
  journal = {Systems \& Control Letters},
  year = {1995},
  volume = {24},
  pages = {351--359},
  number = {5},
  issn = {01676911}
}

@inproceedings{loko2024growth,
  title={Growth conditions to ensure input-to-state stability of time-delay systems under point-wise dissipation},
  author={Loko, Epiphane and Chaillet, Antoine and Wang, Yuan and Karafyllis, Iasson and Pepe, Pierdomenico},
  booktitle = {Proc. of 63rd IEEE Conference on Decision and Control},
  pages={7902--7907},
  year={2024},
  publisher = {IEEE}
}

@article{mancilla2024forward,
  title={Forward completeness does not imply bounded reachability sets and global asymptotic stability is not necessarily uniform for time-delay systems},
  author={Mancilla-Aguilar, Jose L and Haimovich, Hernan},
  journal={Automatica},
  volume={167},
  pages={111788},
  year={2024},
  publisher={Elsevier}
}

@article{hagood2006recovering,
  title={Recovering a function from a Dini derivative},
  author={Hagood, John W and Thomson, Brian S},
  journal={The American Mathematical Monthly},
  volume={113},
  number={1},
  pages={34--46},
  year={2006},
  publisher={Taylor \& Francis}
}

@article{BCMW26,
  author        = {Brivadis, Lucas and Chaillet, Antoine and Mironchenko, Andrii and Wirth, Fabian},
  title         = {Counterexamples in {L}yapunov-{K}rasovskii and input-to-state stability theory},
  year          = {2026},
  journal = {Unpublished report, online at: https://sites.google.com/view/anr-dfg-robusta/publications}
}

@INPROCEEDINGS{MWC24b,
  author={A. Mironchenko and F. Wirth and A. Chaillet and L. Brivadis},
  title = {{ISS} {L}yapunov-{K}rasovskii theorem with point-wise dissipation: a ${V}$-stability approach},
	booktitle = {Proc. of 63rd IEEE Conference on Decision and Control},
  pages = {7896--7901},
  publisher = {IEEE},	
  year = {2024}
}

@INPROCEEDINGS{MIWICHABR24super,
  booktitle = {Proc. of 26th International Symposium on Mathematical Theory of Networks and Systems},
  title={Characterizations of input-to-state stability for time-delay systems},
  author={A. Mironchenko and F. Wirth and A. Chaillet and L. Brivadis},
  pages={512--515},
  year={2024}
}

@book{Mir23,
  title={Input-to-State Stability: Theory and Applications},
  author={Mironchenko, Andrii},
  publisher={Springer Nature},
  year={2023}
}

@PHDTHESIS{Mir12,
  author = {Mironchenko, Andrii},
  title = {Input-to-state stability of infinite-dimensional control systems},
  school = {Department of Mathematics and Computer Science, University of Bremen},
  year = {2012}
}

@article{MiW19a,
  author = {Andrii Mironchenko and Fabian Wirth},
  title = {Non-coercive {L}yapunov functions for infinite-dimensional systems},
  journal = {Journal of Differential Equations},
  year = {2019},
	volume={266},
	number={11},
  pages={7038--7072}
}

@ARTICLE{MiW18b, 
author = {Andrii Mironchenko and Fabian Wirth},
journal={IEEE Transactions on Automatic Control}, 
title={Characterizations of input-to-state stability for infinite-dimensional systems}, 
year={2018}, 
volume={63}, 
number={6}, 
pages={1692-1707}
}

@INPROCEEDINGS{MiW17e,
  author = {Mironchenko, Andrii and Fabian Wirth},
  title = {Input-to-state stability of time-delay systems: Criteria and open problems},
  booktitle = {Proc. of 56th IEEE Conference on Decision and Control},	
  pages={3719--3724},
  year={2017},
  organization={IEEE}	
}

@article{MiP20,
    title={Input-to-state stability of infinite-dimensional systems: Recent results and open questions},
    author={Andrii Mironchenko and Christophe Prieur},
		journal = {SIAM Review},
		volume = {62},
		number = {3},
		pages = {529-614},
		year = {2020},
}

@article{CMW26c,
  author = {Qiaoling Chen and Andrii Mironchenko
            and Fabian Wirth},
  title  = {Lyapunov Theory for Output-to-State
            Stability of Nonlinear
            Infinite-Dimensional Systems},
  journal   = {Submitted. Available online at             arXiv},
  year ={2026}
}

@phdthesis{thesisLoko2025,
    author = {Dagb{\'e}gnon Epiphane Loko},
    title = {Analyse de la Stabilité des Systèmes de Dimension
Infinie et Eﬀets des Perturbations: Théorie et Applications},
    school = {École Nationale des Ponts et Chaussées, Paris},
    year = 2025
}

@article{orlowski2022adaptive,
  title={Adaptive control of {Lipschitz} time-delay systems by sigma modification with application to neuronal population dynamics},
  author={Or{\l}owski, Jakub and Chaillet, Antoine and Destexhe, Alain and Sigalotti, Mario},
  journal={Systems \& Control Letters},
  volume={159},
  pages={105082},
  year={2022},
  publisher={Elsevier}
}

@article{Pepe07-correct,
  title = {On {L}iapunov--{K}rasovskii functionals under {C}arath{\'e}odory conditions},
  author = {Pepe, Pierdomenico},
  journal = {Automatica},
  volume = {43},
  number = {4},
  pages = {701--706},
  year = {2007}
}

@article{pepe2007problem,
  title={The problem of the absolute continuity for {L}yapunov--{K}rasovskii functionals},
  author={Pepe, Pierdomenico},
  journal={IEEE Transactions on Automatic Control},
  volume={52},
  number={5},
  pages={953--957},
  year={2007}
}

@article{driver1962existence,
  title={Existence and stability of solutions of a delay-differential system},
  author={Driver, Rodney D.},
  journal={Archive for Rational Mechanics and Analysis},
  volume={10},
  number={1},
  pages={401--426},
  year={1962}
}

@INPROCEEDINGS{KankWang16, author={Kankanamalage, Hasala Gallolu and Wang, Yuan}, booktitle={2016 IEEE 55th Conference on Decision and Control (CDC)}, title={Lyapunov descriptions of robust output stability for systems with delays}, year={2016}, volume={}, number={}, pages={6416--6421}}



\end{document}